\documentclass{amsart}
\usepackage{amsmath}
\usepackage{graphicx}
\usepackage{hyperref}
\usepackage{amssymb}
\hypersetup{colorlinks=true,citecolor=blue,linkcolor=cyan}
\newtheorem{thm}{Theorem}[section]
\newtheorem{cor}[thm]{Corollary}
\newtheorem{lem}[thm]{Lemma}
\newtheorem{prop}[thm]{Proposition}
\newtheorem{conj}[thm]{Conjecture}
\theoremstyle{definition}

\theoremstyle{remark}
\newtheorem{rem}[thm]{Remark}
\numberwithin{equation}{section}
\let\oldFootnote\footnote
\newcommand\nextToken\relax
\renewcommand\footnote[1]{%
    \oldFootnote{#1}\futurelet\nextToken\isFootnote}
\newcommand\isFootnote{%
    \ifx\footnote\nextToken\textsuperscript{,}\fi}

\newcommand{\F}{\mathcal{F}}
\newcommand{\Ff}{\mathbb{F}}

\newcommand{\Z}{\mathbb{Z}}
\newcommand{\Rr}{\mathbb{R}}
\newcommand{\Hh}{\mathcal{H}}
\newcommand{\Tr}{\mbox{Tr}}

\newcommand{\supp}{\mbox{supp}}

\newcommand{\rad}{\mbox{rad}}
\newcommand{\Res}{\mbox{Res}}

\begin{document}

\title[Low-lying zeros of elliptic curve L-functions over Function Fields]{Low-lying zeros in families of elliptic curve L-functions over Function Fields}%


\author{Patrick Meisner}
\address{KTH Royal Institute of Technology}
\email{pfmeisner@gmail.com}
\author{Anders S\"{o}dergren}
\address{Department of Mathematical Sciences, Chalmers University of Technology and the \newline 
\rule[0ex]{0ex}{0ex}\hspace{9pt} University of Gothenburg, SE-412 96 Gothenburg, Sweden}
\email{andesod@chalmers.se}

\thanks{The first author was supported by the Verg foundation. The second author was supported by a grant from the Swedish Research Council (grant 2016-03759).}

\begin{abstract}

We investigate the low-lying zeros in families of $L$-functions attached to quadratic and cubic twists of elliptic curves defined over $\Ff_q(T)$. In particular, we present precise expressions for the expected values of traces of high powers of the Frobenius class in these families with a focus on the lower order behavior. As an application we obtain results on one-level densities and we verify that these elliptic curve families have orthogonal symmetry type. In the quadratic twist families our results refine previous work of Comeau-Lapointe. Moreover, in this case we find a lower order term in the one-level density reminiscent of the deviation term found by Rudnick in the hyperelliptic ensemble. On the other hand, our investigation is the first to treat these questions in families of cubic twists of elliptic curves and in this case it turns out to be more complicated to isolate lower order terms due to a larger degree of cancellation among lower order contributions.  

\end{abstract}
\maketitle
\section{Introduction}\label{intro}

The Katz--Sarnak heuristics \cite{Katz-Sarnak} is concerned with the distribution of low-lying zeros in families of $L$-functions and predicts that this distribution is determined by a certain random matrix model called the symmetry type of the family (see also \cite{SST} and the references therein). In this paper we are interested in families of $L$-functions attached to elliptic curves, where quantities related to low-lying zeros have been studied for a comparatively long time due to their close relation to the average rank of the elliptic curves in the family via the Birch and Swinnerton-Dyer conjecture; see, e.g., \cite{Brumer,Fiorilli,Goldfeld,Heath-Brown}. For families of elliptic curves defined over the rationals, the systematic investigation of low-lying zeros and the Katz--Sarnak heuristic started with papers by Miller \cite{Miller} and Young \cite{Young1,Young2} in the early 2000's. More recent contributions include \cite{DHP,FPS1,HKS,HMM} that focus on lower order terms in the one-level density and the corresponding predictions of the $L$-functions ratios conjecture.     

In the present paper we investigate families of $L$-functions attached to elliptic curves defined over the function field $\Ff_q(t)$. Recall that additional tools are available in the function field setting since, for example, the Riemann Hypothesis is a celebrated theorem due to Deligne \cite{Deligne}. Already the pioneering work by Katz and Sarnak \cite{Katz-Sarnak} studied low-lying zeros in families of quadratic twists of elliptic curves defined over $\Ff_q(t)$ in the limit where both $q$ and the degree of the twists, i.e.\ the degree of the polynomials that parameterize the family, tend to infinity. However, it is expected that keeping $q$ fixed while the degree of the twists tend to infinity is a closer analogue of quadratic twist families in the number field setting. Important results in this direction were proved by Rudnick \cite{Rudnick} and Bui--Florea \cite{BuiFlorea} who investigated the low-lying zeros of quadratic Dirichlet $L$-functions in the hyperelliptic ensemble. More recently, Comeau-Lapointe \cite{C-L} investigated expected values of traces of high powers of the Frobenius class and the one-level density of families of quadratic twists of elliptic curves in this context and used the results to give upper bounds on the average rank in these families. In  this paper, we refine results in \cite{C-L} to isolate lower order terms and compare the structure of our results with the results of Rudnick \cite{Rudnick} for quadratic Dirichlet $L$-functions. 

In recent years, there has been an increased interest in a variety of different aspects of higher order characters and twists; see, e.g., \cite{BaierYoung,CP1,CP2,C-L,DFK,DFL1,DFL2,DG,Meisner1,Meisner2}. Motivated by this development, we investigate expected values of traces of high powers of the Frobenius class and the one-level density of families of cubic twists of elliptic curves of the form $y^2=x^3+B$ defined over $\Ff_q(t)$. In this case we are not able to isolate lower order terms and we discuss what is needed in order to obtain refined results also in this situation.

We now turn to a precise description of our results.

\subsection{Setup}

Fix a prime $p\not=2,3$ and let $q=p^m$ for some power $m\in\Z_{\geq1}$. For simplicity, assume $q\equiv 1 \bmod{6}$. Let $E$ be an elliptic curve defined over $\Ff_q(T)$ given by the minimal Weierstrass equation $y^2=x^3+Ax+B$, where $A,B\in \Ff_q[T]$. Then the $L$-function attached to $E$ is (cf., e.g., \cite[Lecture 1]{Ulmer})
\begin{align} \label{Lfuncdef}
L(u,E) := \prod_{P|\Delta} \left( 1- a_P(E)u^{\deg(P)} \right)^{-1} \prod_{P\nmid \Delta} \left( 1 - a_P(E)u^{\deg(P)} + u^{2\deg(P)} \right)^{-1},
\end{align}
where $\Delta =\Delta(E) = 4A^3+27B^2$ is the discriminant of $E$ and
$$\#E(P) = q^{\deg(P)}+1  - a_P(E) q^{\deg(P)/2}.$$
Here $E(P)$ is the curve\footnote{If $P$ is a prime of good reduction, then $E(P)$ is in fact an elliptic curve.} obtained from reducing $E$ modulo a prime polynomial $P$ and $\#E(P)$ denotes the number of $\Ff_{q^{\deg(P)}}$-rational points on the non-singular locus of $E(P)$. If it is clear which elliptic curve we are referring to, we will simply write $a_P$ instead of $a_P(E)$. Recall, in particular, that with the above normalization the Hasse--Weil bound states that $|a_P|\leq 2$. Recall also that $L(u,E)$ is a polynomial of degree
$$\mathfrak{n} = \mathfrak{n}_E := \deg(M)+2\deg(\dot A)-4$$
all of whose zeros lie on the ``critical line" $|u|=q^{-1/2}$, where $M$ is the product of prime polynomials with multiplicative reduction and $\dot A$ is the product of prime polynomials with additive reduction (see \cite[Lecture 1]{Ulmer}). Furthermore, $L(u,E)$ satisfies the functional equation
$$L(u,E) = \epsilon(E)(\sqrt{q}u)^{\mathfrak{n}_E}L\left(\frac{1}{qu},E\right),$$
where $\epsilon(E)\in \{\pm1\}$ is the root number of the elliptic curve $E$. 

We are interested in investigating the one-level density of the zeros of these $L$-functions. That is, for an even Schwartz test function $f$, we define the one-level density of $E$ as
$$\mathcal{D}(E,f) := \sum_{\theta} f\left( \mathfrak{n} \frac{\theta}{2\pi} \right),$$
where the sum is over all $\theta\in\Rr$ such that $q^{-1/2}e^{i\theta}$ is a zero of $L(u,E)$, counted with multiplicity. Since $L(u,E)$ is a polynomial with all its zeros on the critical line, we can find a unitary $\mathfrak{n}\times\mathfrak{n}$ matrix $\Theta_E$ (in fact a conjugacy class of unitary matrices), called the Frobenius, such that
\begin{align}\label{FrobDef}
L(u,E) = \det(1-\sqrt{q}u\Theta_E).
\end{align}
Defining the one-level density of a unitary $\mathfrak{n}\times\mathfrak{n}$ matrix $U$ as
$$\mathcal{D}(U,f) = \sum_{j=1}^{\mathfrak{n}} \sum_{n\in\Z} f\left(\mathfrak{n}\left(\frac{\theta_j}{2\pi} -n\right)\right),$$
where $\theta_j$ $(1\leq j\leq \mathfrak{n})$ are the eigenangles of $U$, we immediately get the relation
$$\mathcal{D}(E,f) = \mathcal{D}(\Theta_E,f).$$
Finally, we may apply Poisson summation to obtain, for any unitary $\mathfrak{n}\times\mathfrak{n}$ matrix,
\begin{align}\label{Poisson}
\mathcal{D}(U,f) = \frac{1}{\mathfrak{n}} \sum_{n\in\mathbb{Z}} \widehat{f}\left( \frac{n}{\mathfrak{n}} \right) \Tr(U^n).
\end{align}
Hence, for Schwartz test functions $f$ whose Fourier transforms are supported in $(-\alpha,\alpha)$, to determine the expected value of $\mathcal{D}(E,f)$ as $E$ ranges over some family of elliptic curves, it is enough to determine the expected value of $\Tr(\Theta_E^n)$ for $n<\alpha \mathfrak{n}$.

\subsection{Quadratic twists}

The first family we will be interested in is the family of quadratic twists of a given elliptic curve $E$. That is, if $E: y^2=x^3+Ax+B$ and $D$ is a polynomial, then we define the quadratic twist of $E$ by $D$ as the curve with affine model
$$E_D : y^2 = x^3+AD^2x+BD^3.$$
As $D$ varies, these equations will give distinct elliptic curves if and only if the polynomials $D$ are square-free and coprime to the discriminant $\Delta$ of $E$. Moreover, we see that all the primes that divide $D$ will have additive reduction and that if $D$ is monic then the prime at infinity will have the same reduction type on $E_D$ as it did on $E$ (see Appendix \ref{Section:comments}). Therefore, the degree of $L(u,E_D)$ will be
$$\mathfrak{n}_{E_D} = \mathfrak{n} + 2\deg(D).$$
Hence, if we consider the family of twists coming from the set
$$\Hh_N(\Delta) := \big\{ D\in \Ff_q[T]: D \mbox{ monic, square-free, coprime to $\Delta$ and } \deg(D)=N\big\},$$
then we see that this will form a family of distinct elliptic curves all of whose Frobenius elements have the same size.

For any $D\in\Hh_N(\Delta)$ let $\epsilon(E_D)$ denote the root number of the $L$-function attached to $E_D$. Then (by \cite[Proposition 4.3]{BH}) 
$$\epsilon(E_D) = \epsilon_N\epsilon(E) \chi_D(M),$$
where $\chi_D = \left(\frac{D}{\cdot}\right)$ is the Kronecker symbol and $\epsilon_N = \pm1$, depending only on the value of $\deg(D)=N$. Therefore we define the sets
$$\Hh_N^{+}(\Delta) := \big\{D\in \Hh_N(\Delta) : \chi_D(M) = \epsilon_N\epsilon(E) \big\}$$
and 
$$\Hh_N^{-}(\Delta) := \big\{D\in \Hh_N(\Delta) : \chi_D(M) = -\epsilon_N\epsilon(E) \big\}. $$
We will typically, however, just write $\Hh_N^{\pm}(\Delta)$ to mean either the set $\Hh_N^+(\Delta)$ or the set $\Hh_N^-(\Delta)$. 

It was recently proven by Comeau-Lapointe \cite{C-L} that, for $\epsilon>0$, $n\in\Z_{\geq1}$ and $N> 4\mathfrak{n}+16$, the averages of traces of powers of $\Theta_{E_D}$ satisfy\footnote{Here, and throughout this paper, we use the convention that for any finite and non-empty set $S$ and any function $\phi$ on $S$, $\left\langle \phi\right\rangle_S = \frac{1}{|S|} \sum_{s\in S} \phi(s)$.}\footnote{Note that the family considered in loc. cit. is not the same as the one stated here. However, one may easily deduce this result from that of \cite{C-L}.}
\begin{align}\label{AntRes}
    \big\langle Tr(\Theta^n_{E_D}) \big\rangle_{\Hh^{\pm}_N(\Delta)} = \eta_2(n) + O_{\epsilon}\left( (n+N)N^{2\mathfrak{n}+11}\left(\frac{1}{q^{N/8}} + \frac{1}{q^{\epsilon N}} + \frac{q^{n/2}}{q^{(1-\epsilon)N}}\right) + \frac{n^2}{q^{n/4}} \right),
\end{align}
where (see \cite[Theorem 4]{DS})
$$\eta_2(n) := \int_{\mathrm{O}(\mathfrak{n}+2N)} \Tr(U^n)\,dU = \begin{cases} 1 & 2|n, \\ 0 & 2\nmid n. \end{cases}$$
This is then enough to deduce that if $\supp(\widehat{f}) \subset (-1,1)$, then
\begin{align}\label{AntRes2}
    \big\langle \mathcal{D}(E_D,f) \big\rangle_{\Hh^{\pm}_N(\Delta)}  = \int_{\mathrm{O}(\mathfrak{n}+2N)} \mathcal{D}(U,f)\,dU + O\left(\frac{1}{N}\right).
\end{align}

\begin{rem}
For every sufficiently nice family of elliptic curves $\F$, the Katz--Sarnak heuristic predicts that
$$\big\langle \mathcal{D}(E,f) \big\rangle_{\F} = \int_G \mathcal{D}(U,f)\,dU,$$
where $G$ is a compact Lie group indicating the symmetry type of the family and $dU$ is the Haar measure on $G$. 
Recall that the one-level densities of the three orthogonal symmetry types $\mathrm{O}$, $\mathrm{SO}(even)$ and $\mathrm{SO}(odd)$ agree for test functions whose Fourier transforms are supported in $(-1,1)$. Therefore, for the sake of tidiness, we have chosen to state all results in terms of the symmetry type $\mathrm{O}$.
\end{rem}

\begin{rem}
In \eqref{AntRes} and \eqref{AntRes2} the implied constants depend on $E$. All implied constants in the rest of this paper are similarly allowed to depend on the base elliptic curve. Moreover, throughout this paper we implicitly restrict our attention to non-empty families $\Hh_N^{\pm}(\Delta)$. Note that a family of this type is empty only if $M=1$ so that the root number is constant (see Lemma \ref{pmSizeLem}).
\end{rem}

In this paper, we are interested in determining lower order terms in the estimate \eqref{AntRes}. Specifically, to deduce the exact form of the error term $O(\frac{n^2}{q^{n/4}})$.

\begin{thm}\label{QuadThm}
Let $E$ be an elliptic curve defined over $\Ff_q(T)$ and given by the minimal Weierstrass equation $y^2=x^3+Ax+B$, where $A,B\in \Ff_q[T]$. Let $n\in\Z_{\geq1}$ and assume that $M$ is not a prime of odd degree dividing $n$. Then, for any $\epsilon>0$ and $N> 4\mathfrak{n}+16$, we have
\begin{multline*}
    \big\langle\Tr(\Theta^n_{E_D})\big\rangle_{\Hh^{\pm}_N(\Delta)} = \eta_2(n) \left(1 + \frac{\Tr\big(\Theta^{n/2}_{sym^2E}\big)}{q^{n/4}} + \frac{\mathcal{D}(n)}{q^{n/2}}\right) \\ +  O_{\epsilon}\left(  (n+N)N^{2\mathfrak{n}+11}\left(\frac{1}{q^{N/8}} + \frac{1}{q^{\epsilon N}} + \frac{q^{n/2}}{q^{(1-\epsilon)N}}\right)\right),
\end{multline*}
where $\Theta_{sym^2E}$ is the Frobenius element attached to the symmetric square $L$-function $L(u,sym^2E)$ and $\mathcal{D}(n)$ is given by \eqref{DDef}; in particular, $\mathcal{D}(n) \ll \tau(n)+\deg(\Delta)$.
\end{thm}

\begin{rem}
In the case where $M$ is a prime of odd degree dividing $n$, then the result still holds with the only difference that there is an additional contribution to $\mathcal{D}(n)$. See \eqref{M=P} and the brief discussion thereafter for further details.
\end{rem}

While we are able to improve the error term slightly by finding some secondary terms, we retain the error term containing the expression $\frac{q^{n/2}}{q^{(1-\epsilon)N}}$ and so we are not able to extend the range of $\supp(\widehat{f})$ in \eqref{AntRes2}. However, we are able to write down a term in the one-level density that is reminiscent of the deviation term that Rudnick found for the hyperelliptic ensemble (cf.\ \cite[Corollary 3]{Rudnick}).

\begin{cor}\label{QuadCor}
Let $E$ be an elliptic curve defined over $\Ff_q(T)$ as in Theorem \ref{QuadThm} and let $f$ be an even Schwartz test function. If $\supp(\widehat{f})\subset (-1,1)$, then
$$\big\langle \mathcal{D}(E_D,f)\big\rangle_{\Hh_N^{\pm}(\Delta)} = \int_{\mathrm{O}(\mathfrak{n}+2N)} \mathcal{D}(U,f)\,dU + \frac{dev_E(f)}{N} + O_{\epsilon}\left(\frac{1}{N^{2-\epsilon}}\right),$$
where
\begin{multline*}
     dev_E(f) = \widehat{f}(0)\Bigg(-\frac{1}{q}\frac{L'(q^{-1},sym^2E)}{L(q^{-1},sym^2E)}+ \sum_{P \nmid \Delta} \frac{\deg(P)}{|P|+1} \sum_{d=1}^{\infty}\frac{a^*_{1,P^{2d}}}{|P|^d} \\ - \sum_{P|\Delta} \deg(P) \sum_{d=1}^{\infty} \frac{a^*_{1,P^{2d}} - a^*_{2,P^d}+1}{|P|^d} \Bigg),
\end{multline*}
$a^*_{1,P^d}$ is the $P^d$-th coefficient of $\frac{L'(u,E)}{L(u,E)}$, $a^*_{2,P^d}$ is the $P^d$-th coefficient of $\frac{L'(u,sym^2 E)}{L(u,sym^2E)}$ and $|P|=q^{\deg(P)}$.
\end{cor}

In the family of $L$-functions $L(u,\chi_D)$ attached to quadratic characters, Rudnick \cite[Corollary 3]{Rudnick} showed that the one-level density is asymptotically the same as for the unitary symplectic matrices with a deviation term of the form
$$dev(f) = \widehat{f}(0) \sum_{P} \frac{\deg(P)}{|P|^2-1} - \widehat{f}(1)\frac{1}{q-1}.$$
For $f$ of small support, the main terms in the one-level density come from the prime squares. We note that $\chi_D(P^2) = \chi^2_D(P)$ and thus the contribution of the prime squares to the explicit formula is determined by the logarithmic derivative of
$$L(u,\chi^2_D) = R_D(u) \zeta_q(u),$$
where $R_D$ is a finite Euler product. Taking logarithmic derivatives, there is a simple pole at $u=q^{-1}$. The residue of $\zeta'_q/\zeta_q$ corresponds to the matrix integral whereas one can show that
$$\left\langle \frac{R'_D(q^{-1})}{R_D(q^{-1})} \right\rangle_{\Hh_{2g+1}(1)} = \sum_{P} \frac{\deg(P)}{|P|^2-1}+ O(q^{-g}).$$
To see the term containing $\widehat{f}(1)$ in $dev(f)$, one needs to analyze also the contribution from the primes to the explicit formula. As we need to restrict to functions that have $\widehat{f}(1)=0$, we do not see such a term in Corollary \ref{QuadCor}.

Now, for the quadratic twists of an elliptic curve, we again need to look at the contribution of the prime squares to the explicit formula. Here we obtain terms that contribute to the main term matrix integral, whereas the logarithmic derivative of $$L(u,sym^2E_D) = S_D(u)L(u,sym^2E),$$
where $S_D$ is a finite product of Euler factors, contributes the deviation terms in $dev_E(f)$ (see Section \ref{Subsection:QuadCor}). Note in particular that, similar to the quadratic character case, $L(u,sym^2E_D)$ is essentially constant as we vary $D$, changing only by a finite Euler product. 

Finally, we note the close connection between the formulas in the quadratic character and quadratic twist cases. Indeed, by substituting $1$ for all the $a^*_{1,P^{2d}}$ in the sum over primes of good reduction in $dev_E(f)$, we get exactly the part of $dev(f)$ corresponding to the same set of primes.\footnote{Note that $R_D$ and $\zeta_q(u)$ are related to the constant coefficients $1$ in exactly the same way as $S_D$ and $L(u,sym^2E)$ are related to the coefficients $a^*_{1,P^{2d}}$.}

\subsection{Cubic twists}

Performing a quadratic twist of an elliptic curve has the nice property that if
$L(u,E) = \sum_{F} a_F u^{\deg(F)},$
then
$L(u,E_D) = \sum_{F} a_F \chi_D(F) u^{\deg(F)}.$
Therefore, one natural extension is to consider twists of the $L$-function $L(u,E)$ by other characters, that is, to consider $L$-functions of the form
$$L(u,E,\chi) := \sum_F a_F \chi(F) u^{\deg(F)}.$$
Comeau-Lapointe \cite[Theorem 12.1]{C-L} considered the families where $\chi$ runs over characters of fixed order $\ell\not=2$, and proved that these families have unitary symmetry type. This change of symmetry type is to be expected as when you twist by non-quadratic characters the $L$-functions $L(u,E,\chi)$ fail to be $L$-functions of elliptic curves and therefore loses the orthogonal symmetry inherent in families of elliptic curve $L$-functions.

However, in this paper we choose to twist at the level of elliptic curves instead of at the level of $L$-functions. That is, if the elliptic curve has the special form
$$\widetilde{E}: y^2=x^3+B,$$
with $B\in\Ff_q[T]$, then, for any polynomial $D\in\Ff_q[T]$ coprime to $B$, we define the cubic twist of $\widetilde{E}$ as the curve with affine model
$$\widetilde{E}_D: y^2 = x^3+BD^2.$$
Similar to the case of quadratic twists, as long as $D$ is chosen to be cube-free and coprime to $B$, these will be distinct elliptic curves and all the primes that divide $D$ will have additive reduction. Furthermore, if we consider only the case $3|\deg(D)$, then the prime at infinity  will have the same reduction type on $\widetilde{E}_D$ as it did on $\widetilde{E}$ (see Appendix \ref{Section:comments}) and hence the degree of $L(u,\widetilde{E}_D)$ will be\footnote{Note that here we have to use the radical of $D$. This was not necessary for the quadratic twists since we were assuming $D$ to be square-free in that case and hence equal to its radical.}
$$\mathfrak{n}_{\widetilde{E}_D} = \mathfrak{n}_{\widetilde{E}} + 2\deg(\rad(D)).$$
Therefore, if we define the set 
\begin{multline*}
    \F_{N}(B) := \big\{D\in\Ff_q[T]: D \mbox{ monic, cube-free, } (D,B)=1,\\  \deg(\rad(D)) = N, \deg(D) \equiv 0 \bmod{3}\big\},
\end{multline*}
then the cubic twists by this family will form a family of distinct elliptic curves all of whose Frobenius elements have the same size.

\begin{thm}\label{CubicOLD}
Let $\widetilde{E}$ be an elliptic curve defined over $\Ff_q(T)$ and given by the minimal Weierstrass equation $y^2 = x^3+B$, where $B\in \Ff_q[T]$. Then, for any $\epsilon>0$ and $n\in\Z_{\geq1}$, we have  
$$\big\langle \Tr(\Theta^n_{\widetilde{E}_D})\big\rangle_{\F_{N}(B)} = \eta_2(n) + O_{\epsilon}\left(\frac{q^{n/2}e^{2n}}{Nq^{(\frac{1}{2}-\epsilon)N}} + \frac{\eta_2(n)n}{q^{n/4}} + \frac{1}{q^{n/3}} + \frac{n(\deg(\Delta)+\tau(n))} {q^{n/2}}\right),$$
where $\tau$ is the number of divisors function. Moreover, for any Schwartz test function satisfying $\supp(\widehat{f}) \subset \left(-\alpha,\alpha\right)$ for some $\alpha<\frac{1}{2} -\frac{2}{4+\log q}$, we get
\begin{align}\label{Cubic1level}
\big\langle \mathcal{D}(\widetilde{E}_D,f) \big\rangle_{\F_{N}(B)}  = \int_{\mathrm{O}(\mathfrak{n}+2N)} \mathcal{D}(U,f)\,dU + O\left(\frac{1}{N}\right).
\end{align}
\end{thm}

Similar to Corollary \ref{QuadCor}, the main term in \eqref{Cubic1level} comes from considering the prime squares whose contribution is the sum of the matrix integral and a term determined by the logarithmic derivative of $L(u,sym^2\widetilde{E}_D)$. However, unlike the quadratic twist family, $L(u,sym^2\widetilde{E}_D)$ is not essentially constant as $D$ varies (see Section \ref{Section:cubic}) and thus we get some cancellation that prevents us from obtaining a deviation term. 

The next obvious thing to consider is the contribution from the prime cubes. Using Lemma \ref{PowerReduce}, we find that the contribution from the prime cubes is determined by the logarithmic derivative of
$L(u,sym^3 \widetilde{E}_D)L(u,\widetilde{E}_D)^{-1}.$
However, since the coefficients of $L(u,\widetilde{E}_D)$ are not obtained by a simple twist of a character (as in the quadratic twist case), we get that  $L(u,sym^3\widetilde{E}_D)$ is still not essentially constant as $D$ varies.  Although, $L(u,sym^2\widetilde{E}_D)$ and $L(u,sym^3\widetilde{E}_D)$ will have parts that are essentially constant.

In order to describe these essentially constant parts, we need to introduce some notation. For any prime $P$ and any $\widetilde{E}$, define
$$\lambda_P = \lambda_P(\widetilde{E}) := \frac{1}{q^{\deg(P)/2}} \sum_{F\bmod P} \left(\frac{F^2-B}{P} \right)_3,$$
where $\left(\frac{\cdot}{P}\right)_3$ is the cubic residue symbol modulo $P$.\footnote{We assume $q\equiv 1 \bmod{6}$ so that the cubic residue symbol is well defined.} The quantity $\lambda_P$ behaves nicely with respect to cubic twists. Namely,
\begin{align}\label{TwistedlambdaP}
\lambda_P(\widetilde{E}_D) = \left(\frac{D}{P}\right)^2_3 \lambda_P(\widetilde{E})
\end{align}
(cf.\ Lemma \ref{Lem:cubictwistcoeff}). Moreover, we get that $\lambda_P+\overline{\lambda}_P = -a_P$, where as usual $a_P$ denotes the $P$-th coefficient of $L(u,\widetilde{E})$. Therefore, it is possible to write $a^*_{1,P^d}$ in terms of the $\lambda_P$ and then determine how these vary with $D$. However, we can show that approximately half of the time $|\lambda_P|\not=1$ (see Corollary \ref{maincor}), and hence $\lambda_P$ can typically not be a root of $1-a_Pu+u^2$. Therefore, while possible, writing $a^*_{1,P^d}$ in terms of $\lambda_P$ is in general not so nice. Using Lemma \ref{atolambda} and \eqref{TwistedlambdaP} to identify parts of $a^*_{1,P^2}(\widetilde{E}_D)$ and $a^*_{1,P^3}(\widetilde{E}_D)$ that are essentially constant as we vary $D$, we get the following theorem.

\begin{thm}\label{CubicLowTerm}
Let $\widetilde{E}$ be an elliptic curve defined over $\Ff_q(T)$ and given by the minimal Weierstrass equation $y^2 = x^3+B$, where $B\in \Ff_q[T]$. Then, for any $\epsilon>0$ and $n\in\Z_{\geq1}$, we have  
\begin{multline*}
    \big\langle \Tr(\Theta^n_{\widetilde{E}_D}) \big\rangle_{\F_{N}(B)} 
    = -\frac{n/2}{q^{n/2}} \sum_{\deg(P)=\frac{n}{2}} 2(|\lambda_P|^2-1)  +\frac{n/3}{q^{n/2}} \sum_{\deg(P)=\frac{n}{3}} \big(\lambda_P^3 + \overline{\lambda}^3_P\big)\\ 
    + \frac{\widetilde{\mathcal{D}}_1(n)+\widetilde{\mathcal{D}}_2(n)}{q^{n/2}} + O_{\epsilon}\left(\frac{q^{n/2}e^{2n}}{Nq^{(\frac{1}{2}-\epsilon)N}}+\frac{1}{q^{3n/8}} + \frac{n(\deg(\Delta)+\tau(n))}{q^{n/2}} + \frac{1}{Nq^{n/2}}\right), 
\end{multline*}
where $\tau$ is the number of divisors function and $\widetilde{\mathcal{D}}_1(n)$ and $\widetilde{\mathcal{D}}_2(n)$ are defined in \eqref{CubicDForm1} and \eqref{CubicDform2}; in particular, $\widetilde{\mathcal{D}}_1(n),\widetilde{\mathcal{D}}_2(n) \ll 1$.
\end{thm}

\begin{rem}\label{Remark:tertiarymainterm}
Note that the presence of the terms $q^{-3n/8}$ and $n\tau(n)q^{-n/2}$ in the error term means that we could absorb the tertiary main term $(\widetilde{\mathcal{D}}_1(n)+\widetilde{\mathcal{D}}_2(n))q^{-n/2}$ into the error term. However, these error terms come from trivially bounding primes of degree at most $\frac{n}{4}$ and the primes of bad reduction, respectively, while $\widetilde{\mathcal{D}}_1(n)$ is written as a sum over primes of degree $\frac{n}{2}$ and $\widetilde{\mathcal{D}}_2(n)$ is written as a sum over primes of degree $\frac{n}{3}$. Therefore, if one is careful, one would (at least for small $n$) be able to remove these error terms and make the tertiary main term a sum over all primes of degree dividing $n$ and a real main term. This process would involve extending Lemma \ref{atolambda}. While not difficult, it would result in a less clean statement of Theorem \ref{CubicLowTerm}. Note also that this new tertiary main term may no longer be bounded by $q^{-n/2}$.
\end{rem}

It is not clear that Theorem \ref{CubicLowTerm} gives us what we were hoping for. That is, a term of size $q^{-n/3}$ that is related to the Frobenius of $L(u,sym^3\widetilde{E})$ in some way which would lead to a deviation term for the one-level density that involves the logarithmic derivative of $L(u,sym^3\widetilde{E})$. However, using some heuristic arguments, we can see a way that this appears.

\subsection{Heuristics and conjectures}\label{Intro-Heur-Conj}

Using the Weil bound and comparing Theorems \ref{CubicOLD} and \ref{CubicLowTerm}, we get the following immediate corollary:

\begin{cor}\label{maincor}
For any $m\in\Z_{\geq1}$, we have
$$\frac{m}{q^m} \sum_{\deg(P)=m} |\lambda_P|^2 = \frac{1}{2} + O\left(\frac{1}{q^{m/3}}\right).$$
\end{cor}

Heuristically, replacing $|\lambda_P|^2$ with its average of $\frac{1}{2}$, we can show that 
\begin{align*}
    \frac{n/3}{q^{n/2}}\sum_{\deg(P)=\frac n3} \big(\lambda_P^3+ \overline{\lambda}_P^3\big) \approx \frac{\eta_3(n)}{q^{n/3}} \left( \Tr\big(\Theta^{n/3}_{sym^3 \widetilde{E}}\big) +\frac{1}{2} \Tr\big(\Theta^{n/3}_{\widetilde{E}}\big)\right)
\end{align*}
(see Section \ref{Section:heuristics}), where
$$\eta_3(n) := \begin{cases} 1 & 3|n, \\ 0 & 3\nmid n. \end{cases}$$
This naturally leads to the following conjecture:

\begin{conj}\label{CubicConj}
Let $\widetilde{E}$ be an elliptic curve defined over $\Ff_q(T)$ and given by the minimal Weierstrass equation $y^2 = x^3+B$, where $B\in \Ff_q[T]$. Then, for any $n\in\Z_{\geq1}$, we have  
\begin{align*}
\big\langle\Tr(\Theta^n_{\widetilde{E}_D})\big\rangle_{\F_{N}(B)} = \eta_2(n) + \frac{\eta_3(n)}{q^{n/3}}\left( \Tr\big(\Theta^{n/3}_{sym^3 \widetilde{E}}\big) +\frac{1}{2} \Tr\big(\Theta^{n/3}_{\widetilde{E}}\big)\right) 
+\frac{\widetilde{\mathcal{D}}(n)}{q^{n/2}}\big(1+o(1)\big),
\end{align*}
where $\widetilde{\mathcal{D}}(n)$ can be written as a sum over primes of degree dividing $n$ and is bounded by (and might be considerably smaller than) $q^{n/8}$.
\end{conj}

Therefore, passing to the one-level density, we get an idea of what kind of deviation term we could expect in cubic twist families. As this relies on the conjecture and the proof would be essentially the same as the one of Corollary \ref{QuadCor}, we state only that the deviation term should contain the two terms
$$-\widehat{f}(0)\frac{L'(q^{-3/2},sym^3\widetilde{E})}{q^{3/2}L(q^{-3/2},sym^3\widetilde{E})} \quad \quad \mbox{ and } \quad \quad -\widehat{f}(0) \frac{L'(q^{-3/2},\widetilde{E})}{2q^{3/2}L(q^{-3/2},\widetilde{E})},$$
as well as a third term which can be expressed as a sum over primes and will depend on $\widetilde{\mathcal{D}}(n)$.

\subsection*{Outline of the paper}

In Section \ref{Section:symmetric}, we define the symmetric power $L$-functions and prove several relations that will be needed throughout the later sections. In Section \ref{Section:quadratic}, we briefly discuss the quadratic twist family as the majority of the work in getting a reasonable error term is already done in \cite{C-L}. In Section \ref{Section:cubic}, we consider the cubic twist family and prove Theorems \ref{CubicOLD} and \ref{CubicLowTerm}. In Section \ref{Section:heuristics}, we give a heuristic argument for Conjecture \ref{CubicConj}. Finally, in Appendix \ref{Section:comments} we discuss the choices we made in choosing our families and indicate how one could extend our work to ``fuller" families.

\subsection*{Acknowledgments}

We thank Daniel Fiorilli and Ze\'{e}v Rudnick for helpful comments on an early version of this paper. We also thank Lucile Devin for pointing our attention to Lumley’s paper \cite{Lumley}.

\section{Symmetric power $L$-functions}\label{Section:symmetric}

Let $E$ be any elliptic curve defined over $\Ff_q(T)$. Recall the definition of the $L$-function attached to $E$
\begin{align*}
L(u,E) := \prod_{P|\Delta} \left( 1- a_P(E)u^{\deg(P)} \right)^{-1} \prod_{P\nmid \Delta} \left( 1 - a_P(E)u^{\deg(P)} + u^{2\deg(P)} \right)^{-1},
\end{align*}
which converge for $|u|<q^{-1}$. If $P\nmid \Delta$, then we define $\alpha_P$ and $\beta_P$ such that
\begin{align}\label{alphaDef}
    1-a_Pu+u^2 = (1-\alpha_Pu)(1-\beta_Pu).
\end{align}
In addition, when $P|\Delta$, we set $\alpha_P=a_P$ and $\beta_P=0$.
In other words, we define $\alpha_P$ and $\beta_P$ such that the inverse of the Euler factor at $P$ equals
$$L_P(u,E) = \big(1-\alpha_Pu^{\deg(P)}\big)\big(1-\beta_Pu^{\deg(P)}\big),$$
where $\beta_P=0$ if $E$ has bad reduction at $P$.

For $m$ a positive integer, we define
\begin{align}\label{SymmPowDef}
L(u,sym^mE) := \prod_{P\nmid \Delta} \prod_{i=0}^m \left(1-\alpha_P^i\beta_P^{m-i}u^{\deg(P)}\right)^{-1} \prod_{P|\Delta} L_P\big(u^{\deg(P)},sym^mE\big)^{-1},
\end{align}
where $L_P(u,sym^mE)$ is a polynomial of degree at most $m+1$ with bounded coefficients.\footnote{In fact, if $P$ has multiplicative reduction, then $L_P(u,sym^mE) = (1-\alpha_P^mu)$. However, if $P$ has additive reduction, then the situation is more complicated.} We refer the reader to \cite[Section 1.2]{CFJ} for more information on symmetric power $L$-functions, and the references therein (specifically \cite{Deligne2} and \cite{Ulmer-geometric}) for more general statements and proofs. See also \cite{Martin-Watkins} for symmetric power $L$-functions of elliptic curves defined over $\mathbb{Q}$.

\begin{prop}[Parts of Theorem 1.1 of \cite{CFJ}]\label{RHProp}
For any elliptic curve $E$ defined over $\Ff_q(T)$ and any positive integer $m$, $L(u,sym^mE)$ is a polynomial of degree $\mathfrak{n}_m$ all of whose roots have norm $q^{-1/2}$. Hence, we can find a matrix $\Theta_{sym^mE} \in \mathrm{U}(\mathfrak{n}_m)$ such that
$$L(u,sym^mE) = \det\big(1-\sqrt{q}u\Theta_{sym^mE}\big).$$
\end{prop}

\begin{rem}\label{Remark:CFJ}
I follows from \cite[Lemma 2.1]{CFJ-PNR} that $\mathfrak{n}_m\ll m$ for all positive integers $m$, where the implied constant depends on $E$.
\end{rem}

It will be useful to have notation for $L(u,sym^mE)$ also when $m=0$ and $-1$. Therefore, we define
$$L(u,sym^0E) := \zeta_q(u) \quad \quad \mbox{ and } \quad \quad L(u,sym^{-1}E) := 1,$$
where $\zeta_q(u)$ is the usual zeta function of $\Ff_q[T]$ defined as
$$\zeta_q(u) := \sum_{F \,\,\text{monic}} u^{\deg(F)} = \frac{1}{1-qu}.$$

\subsection{A symmetric power trace formula}

We define $a^*_{m,P^k}=a^*_{m,P^k}(E)$ such that
\begin{align}\label{amPkDef}
    \frac{L'(u,sym^mE)}{L(u,sym^mE)} = \frac{1}{u}\sum_P \deg(P) \sum_{k=1}^{\infty} a^*_{m,P^k} u^{k\deg(P)}.
\end{align}
If $m$ is a positive integer, then we can use \eqref{amPkDef} together with Proposition \ref{RHProp} to get a formula for the trace of the Frobenius element:
\begin{align}\label{SymmEllTraceForm}
    -q^{n/2}\Tr(\Theta^n_{sym^mE}) = \sum_{d|n}\frac{n}{d} \sum_{\deg(P)=\frac{n}{d}} a^*_{m,P^d}.
\end{align}
If $m$ is not a positive integer, then we see that 
\begin{align}\label{a0PkDef}
    a^*_{0,P^k}=1 \mbox{ and } a^*_{-1,P^k}=0
\end{align}
for all $P$ and $k$. Thus, even though there is no Frobenius matrix associated with $m=0,-1$, we can still determine their respective sums in the right-hand side of \eqref{SymmEllTraceForm}:
\begin{align}\label{TraceFormZero}
    \sum_{d|n} \frac{n}{d}\sum_{\deg(P)=\frac{n}{d}} a^*_{m,P^d} = \begin{cases} q^n & m=0, \\ 0 & m=-1.\end{cases}
\end{align}

\subsection{Useful lemmas}

Combining equations \eqref{alphaDef},\eqref{SymmPowDef} and \eqref{amPkDef}, we see that if $m$ is a positive integer and $P$ is a prime of good reduction, then
\begin{align}\label{amPkform}
    a^*_{m,P^k} = \eta_2(m) + \sum_{j=0}^{\lfloor\frac{m-1}{2}\rfloor} \left(\alpha_P^{k(m-2j)}+\beta_P^{k(m-2j)}\right).
\end{align}
Further, for any prime $P$, we have the bound 
\begin{align}\label{amPkBound}
|a^*_{m,P^k}| \leq 2(m+1).
\end{align}
Note that by \eqref{a0PkDef}, we see that \eqref{amPkform} and \eqref{amPkBound} also hold for $m=0,-1$.

We can now use \eqref{amPkform} to relate the coefficients of the logarithmic derivatives of different symmetric power $L$-functions.

\begin{lem}\label{PowerReduce}
If $P$ is a prime of good reduction and $m$ is a positive integer, then 
$$a^*_{1,P^{md}} = a^*_{m,P^d} - a^*_{m-2,P^d}.$$
\end{lem}

\begin{proof}
Applying \eqref{amPkform}, we get
\begin{align*}
    a^*_{m,P^d}-a^*_{m-2,P^d} & = \sum_{j=0}^{\lfloor\frac{m-1}{2}\rfloor} \left(\alpha_P^{d(m-2j)}+\beta_P^{d(m-2j)}\right) - \sum_{j=0}^{\lfloor\frac{m-3}{2}\rfloor} \left(\alpha_P^{d(m-2-2j)}+\beta_P^{d(m-2-2j)}\right) \\
    & = \alpha_P^{dm}+\beta_P^{dm} \\ 
    & = a^*_{1,P^{md}},
\end{align*}
which is the desired result. 
\end{proof}

Next, we are able to use Lemma \ref{PowerReduce} to obtain a nice formula that relates traces of different symmetric powers. But first, we need to introduce some notation that will be useful in order to optimize the contribution of the primes of bad reduction to our error terms. For any $D\in\Ff_q[T]$, and any $n>0$, we denote
\begin{align}\label{D_n}
    D_n = \prod_{\substack{P|D \\ \deg(P)=n}} P.
\end{align}

\begin{lem}\label{NiceTraceFormula}
Let $E$ be any elliptic curve defined over $\Ff_q(T)$. 
\begin{enumerate}
\item Let $m|n$. Then, if $m\geq 3$, we have
$$\sum_{\substack{d|n \\ m|d }} \frac{n}{d} \sum_{\deg(P)=\frac{n}{d}} a^*_{1,P^d} = -q^{\frac{n}{2m}} \left(\Tr\big(\Theta^{n/m}_{sym^mE}\big)  - \Tr\big(\Theta^{n/m}_{sym^{m-2}E}\big) \right)+ O\bigg(m \sum_{d|\frac{n}{m}} \deg(\Delta_{n/dm})  \bigg).$$
\item If $2|n$, then
$$\sum_{\substack{d|n \\ 2|d }} \frac{n}{d} \sum_{\deg(P)=\frac{n}{d}} a^*_{1,P^d} = -q^{n/2}-q^{n/4} \Tr\big(\Theta^{n/2}_{sym^2E}\big) + \sum_{\substack{d|n \\ 2|d }}\frac{n}{d} \sum_{\substack{\deg(P)=\frac{n}{d} \\ P|\Delta }} \big(a^*_{1,P^{d}}-a^*_{2,P^{d/2}}+1 \big). $$
\end{enumerate}
\end{lem}

\begin{proof}
To prove $(1)$, we first observe that 
$$\sum_{\substack{d|n \\ m|d }} \frac{n}{d} \sum_{\deg(P)=\frac{n}{d}} a^*_{1,P^d} = \sum_{d|\frac{n}{m}} \frac{n}{dm} \sum_{\deg(P)=\frac{n}{dm}} a^*_{1,P^{md}}. $$
Splitting the sum over primes into primes of good and bad reduction, we find that the primes of bad reduction contribute
$$ \sum_{d|\frac{n}{m}} \frac{n}{dm} \sum_{\substack{\deg(P)=\frac{n}{dm} \\P|\Delta}} a^*_{1,P^{md}} \ll \sum_{d|\frac{n}{m}} \deg(\Delta_{n/dm}).$$
Now, when $m\geq 3$, for the primes of good reduction, we use Lemma \ref{PowerReduce} as well as \eqref{SymmEllTraceForm} to get
\begin{multline*}
 \sum_{d|\frac{n}{m}} \frac{n}{dm} \sum_{\substack{\deg(P)=\frac{n}{dm} \\ P\nmid \Delta }} a^*_{1,P^{md}} = \sum_{d|\frac{n}{m}} \frac{n}{dm} \sum_{\substack{\deg(P)=\frac{n}{dm} \\ P\nmid \Delta }} \left( a^*_{m,P^d} - a^*_{m-2,P^d}\right)\\
 =- q^{\frac{n}{2m}} \left(\Tr\big(\Theta^{n/m}_{sym^mE}\big) -\Tr\big(\Theta^{n/m}_{sym^{m-2}E}\big) \right) +O\bigg(m \sum_{d|\frac{n}{m}} \deg(\Delta_{n/dm})  \bigg),
\end{multline*}
where the error term again comes from the primes of bad reduction.

Finally, to prove $(2)$ we can still use Lemma \ref{PowerReduce} on the primes of good reduction. However, we also have to use \eqref{a0PkDef} and \eqref{TraceFormZero} in conjunction with \eqref{SymmEllTraceForm} to get
\begin{align*}
 \sum_{d|\frac{n}{2}} \frac{n}{2d} \sum_{\substack{\deg(P)=\frac{n}{2d} \\ P\nmid \Delta }} a^*_{1,P^{2d}} & = \sum_{d|\frac{n}{2}} \frac{n}{2d} \sum_{\substack{\deg(P)=\frac{n}{2d} \\ P\nmid \Delta}} \big( a^*_{2,P^d} - 1\big)\\
 & = -q^{n/2} - q^{n/4} \Tr\big(\Theta^{n/2}_{sym^2E}\big) - \sum_{\substack{d|n \\ 2|d }} \frac{n}{d} \sum_{\substack{\deg(P)=\frac{n}{d} \\ P|\Delta }} \big(a^*_{2,P^{d/2}}-1\big),
\end{align*}
and the result follows.
\end{proof}

We can now use these relations to bound sums of $a^*_{m,P^k}$ over primes $P$ of a fixed degree for various combinations of $m$ and $k$.

\begin{lem}\label{RHbound}
Let $E$ be any elliptic curve defined over $\Ff_q(T)$.
\begin{enumerate}
\item If $m$ is any positive integer, then
$$\sum_{\deg(P)=n} a^*_{m,P} \ll \frac{m}{n} q^{n/2}.$$
\item If $m\geq 3$, then 
$$\sum_{\deg(P)=n}a^*_{1,P^m}\ll \frac{m}{n}\big(q^{n/2} + \deg(\Delta_n)\big).$$
\item For prime squares, we have
$$\sum_{\deg(P)=n}a^*_{1,P^2}= -\frac{q^n}{n} +O\bigg(q^{n/2} + \frac{\deg(\Delta_n)}{n} \bigg).$$
\end{enumerate}
\end{lem}

\begin{proof}
To prove $(1)$, we see from \eqref{SymmEllTraceForm} that
$$\sum_{\deg(P)=n} a^*_{m,P} = -\frac{q^{n/2}}{n} \Tr\big(\Theta^n_{sym^mE}\big) - \sum_{\substack{d|n\\d>1}} 
\frac{1}{d} \sum_{\deg(P)=\frac{n}{d}} a^*_{m,P^d}.$$
Since $\Theta_{sym^mE}$ is a unitary matrix of size $\mathfrak{n}_m\times \mathfrak{n}_m$, we get $|\Tr(\Theta^n_{sym^mE})|\leq \mathfrak{n}_m\ll m$ by Remark \ref{Remark:CFJ}. Moreover, for the prime sum, we use the bound in \eqref{amPkBound} and bound the number of primes of degree $\frac{n}{d}$ by $\frac{q^{n/d}}{n/d}$ to obtain the result.

Now, for $(2)$, we apply Lemma \ref{PowerReduce} and $(1)$ to get
\begin{align*}
    \sum_{\deg(P)=n}a^*_{1,P^m} & = \sum_{\deg(P)=n } \left(a^*_{m,P} - a^*_{m-2,P}\right) + O\Big(\frac{m}{n}\deg(\Delta_n) \Big)
    \ll \frac{m}{n}\big(q^{n/2} + \deg(\Delta_n)\big),
\end{align*}
where the error term $O\left(\frac{m}{n}\deg(\Delta_n)\right)$ comes from using \eqref{amPkBound} for the primes of degree $n$ that divide $\Delta$.

Finally, for $(3)$, we still apply  Lemma \ref{PowerReduce} but now we can only apply $(1)$ on the sum of $a^*_{2,P}$ and need to use \eqref{TraceFormZero} on the sum of $a^*_{0,P}=1$. That is, we get
\begin{align*}
    \sum_{\deg(P)=n}a^*_{1,P^2} & = \sum_{\deg(P)=n } \left(a^*_{2,P} - a^*_{0,P}\right) + O\bigg(\frac{\deg(\Delta_n)}{n}\bigg)\\
    & =-\frac{q^n}{n} + O\bigg(q^{n/2} + \frac{\deg(\Delta_n)}{n}\bigg),
\end{align*}
which completes the proof.
\end{proof}

These first few lemmas are crucial in picking out the lower order terms in the family of quadratic twists. However, they are a little less useful for the family of cubic twists. In that case, we will need the following two lemmas.

\begin{lem}\label{TruncTrForm}
Let $E$ be any elliptic curve defined over $\Ff_q(T)$. For any $m\geq 2$, we have
$$-q^{n/2}\Tr(\Theta_E^n) = \sum_{\substack{d|n\\d\leq m}} \frac{n}{d} \sum_{\deg(P)=\frac{n}{d}} a^*_{1,P^d} + O\Bigg(mq^{\frac{n}{2(m+1)}} + n\sum_{\substack{d|n \\ d>m}} \deg(\Delta_{n/d})\Bigg).$$
\end{lem}

\begin{proof}
We see from \eqref{SymmEllTraceForm} that it suffices to bound $$\sum_{\substack{d|n\\d> m}} \frac{n}{d} \sum_{\deg(P)=\frac{n}{d}} a^*_{1,P^d}.$$
Applying Lemma \ref{RHbound}, we bound this sum by
$$\sum_{\substack{d|n\\d> m}} \left(dq^{\frac{n}{2d}} + d\deg(\Delta_{n/d})\right)\ll mq^{\frac{n}{2(m+1)}} + n\sum_{\substack{d|n \\ d>m}} \deg(\Delta_{n/d}),$$
and the claimed estimate follows. 
\end{proof}

Finally, we note that if we fix an elliptic curve $\widetilde{E} : y^2=x^3+B$ and perform a cubic twist by $D$ for some $D\in\F_N(B)$, then we get that 
$$\deg\big(\Delta(\widetilde{E}_D)_{n/d}\big) \ll \deg\big(\Delta(\widetilde{E})_{n/d}\big) + \deg(D_{n/d}).$$
Thus, in Section \ref{Section:cubic}, we will need a bound of a modified expected value of $\deg(D_{n/d})$ as $D$ ranges over $\F_N(B)$.

\begin{lem}\label{BadPrimeBound}
We have
$$\frac{1}{|\F_N(B)|} \sum_{D\in\F_B(N)} n \sum_{d|n} \deg(D_{n/d}) \ll n\tau(n),$$
where $\tau$ is the number of divisors function.
\end{lem}

\begin{proof}
We have that
\begin{align*}
   \frac{1}{|\F_N(B)|} \sum_{D\in\F_B(N)} n \sum_{d|n} \deg(D_{n/d}) & = \frac{1}{|\F_N(B)|} \sum_{D\in\F_B(N)} n \sum_{d|n} \sum_{\substack{P|D \\ \deg(P)=\frac{n}{d}}} \deg(P) \\ 
   & =  n\sum_{d|n} \sum_{\deg(P)=\frac{n}{d}} \deg(P) \Bigg(\frac{\big|\{D \in \F_N(B) : P|D \}\big|}{|\F_N(B)|}  \Bigg)\\ 
   & \ll n \sum_{d|n} \sum_{\deg(P)=\frac{n}{d}} \frac{\deg(P)}{q^{\deg(P)}}  \ll n \sum_{\substack{ d|n \\ d\geq \frac{n}{N} }} 1 \ll n\tau(n),
\end{align*}
where we have used Remark \ref{Sec2Rem} to bound $\frac{|\{D \in \F_N(B) : P|D \}|}{|\F_N(B)|}$.
\end{proof}

\section{Quadratic twists}\label{Section:quadratic}

In this section, we prove Theorem \ref{QuadThm} and Corollary \ref{QuadCor}.

\subsection{A formula for $a^*_{1,P^k}(E_D)$}

Recall that we are considering an elliptic curve given by the equation
\begin{align*}
E: y^2=x^3+Ax+B,    
\end{align*} 
where $A,B\in\Ff_q[T]$, and that for every $D\in\Hh^{\pm}_N(\Delta)$, we have the quadratic twist
\begin{align}\label{E-eq}
E_D : y^2=x^3+AD^2x+BD^3.
\end{align}
While it is well known how $a^*_{1,P^k}(E_D)$ behaves as we vary $D$, we will prove it here to illustrate the differences between the quadratic twists and the cubic twists (cf.\ Section \ref{Subsection:cubiccoeff}). 

\begin{lem}\label{NumPtsLem}
For any elliptic curve $E$ with discriminant $\Delta$, prime  $P$ and $D\in\Hh^{\pm}_{N}(\Delta)$, we have
$$a_P(E_D) = \left(\frac{D}{P}\right) a_P(E),$$
where for any $F,G\in\Ff_q[T]$, $\left(\frac{F}{G}\right)$ is the quadratic residue symbol.
\end{lem}

\begin{proof}
Recall that $a_P(E_D)$ is defined such that the relation
\begin{align}\label{NumPtsForm}
\#E_D(P) = q^{\deg(P)}+1 - a_P(E_D)q^{\deg(P)/2}
\end{align}
holds. Let us compute $\#E_D(P)$. Since $E_D(P)$ is a curve given by the cubic equation \eqref{E-eq}, reduced modulo $P$, we get that there is exactly one point lying above the point at infinity. For the finite points $F\in \Ff_q[T]/(P)$, the number of points lying above $F$ on $E_D(P)$ is
$$\begin{cases}
2 & \text{if $F^3+AD^2F+BD^3$ is a non-zero square mod $P$,}\\
1 & \text{if $F^3+AD^2F+BD^3\equiv 0$ mod $P$,}\\
0 & \text{if $F^3+AD^2F+BD^3$ is a non-square mod $P$.}
\end{cases}$$
Therefore, we may capture the number of points on $E_D(P)$ as a character sum:
\begin{align}\label{numpts}
\#E_D(P) & = 1 + \sum_{F\bmod{P}} \left(1 + \left(\frac{F^3+AD^2F+BD^3}{P}\right)\right),
\end{align}
where the first term in the right-hand side is the contribution from the point lying over the point at infinity.

Now, if $P\nmid D$, then for every $F\bmod{P}$, we can find a unique $G\bmod{P}$ such that $F=GD$. Hence
\begin{align*}
\#E_D(P) & = 1 + \sum_{G\bmod{P}} \left(1 + \left(\frac{(GD)^3+AD^2(GD)+BD^3}{P}\right)\right)\\
& = q^{\deg(P)} +1 + \left(\frac{D}{P}\right)\sum_{G\bmod{P}}\left(\frac{G^3+AG+B}{P}\right) \\
& = q^{\deg(P)} +1 - \left(\frac{D}{P}\right)a_P(E)q^{\deg(P)/2}.
\end{align*}
Comparing this to \eqref{NumPtsForm} completes the proof for all primes $P\nmid D$. 

On the other hand, if $P|D$, then we see that \eqref{numpts} becomes
\begin{align*}
\#E_D(P) & =1+ \sum_{F\bmod{P}} \left(1 + \left(\frac{F}{P}\right)\right) = q^{\deg(P)}+1.
\end{align*}
It follows that $a_P(E_D)=0=\left(\frac{D}{P}\right)a_P(E)$, which concludes the proof.
\end{proof}

Lemma \ref{NumPtsLem} has the following immediate consequence.

\begin{cor}\label{alphaLem}
With $\alpha_P(E_D)$ and $\beta_P(E_D)$ defined as in \eqref{alphaDef}, we have
$$\alpha_P(E_D) = \left(\frac{D}{P}\right)\alpha_P(E) \hspace{8pt}\mbox{ and }\hspace{8pt} \beta_P(E_D) = \left(\frac{D}{P}\right)\beta_P(E).$$
Consequently, 
$$a^*_{1,P^k}(E_D) = \left(\frac{D}{P}\right)^ka^*_{1,P^k}(E).$$
\end{cor}

\begin{proof}
If $P\nmid D\Delta$, then we get by Lemma \ref{NumPtsLem} that
$$1-a_P(E_D)u+u^2 = \left(1-\alpha_P(E)\left(\frac{D}{P}\right)u\right) \left(1-\beta_P(E)\left(\frac{D}{P}\right)u\right),$$
so that
$$ \alpha_P(E_D) = \left(\frac{D}{P}\right) \alpha_P(E) \hspace{8pt}\mbox{ and }\hspace{8pt} \beta_P(E_D) = \left(\frac{D}{P}\right)\beta_P(E). $$
Hence, by \eqref{amPkform} we obtain
\begin{align*}
    a^*_{1,P^k}(E_D) = \alpha^k_P(E_D)+\beta^k_P(E_D) = \left(\frac{D}{P}\right)^k a^*_{1,P^k}(E).
\end{align*}
Moreover, for $P|D\Delta$, we get by \eqref{amPkDef} and Lemma \ref{NumPtsLem} that
$$a^*_{1,P^k}(E_D) = (a_{P}(E_D))^k= \left(\frac{D}{P}\right)^k (a_{P}(E))^k = \left(\frac{D}{P}\right)^k a^*_{1,P^k}(E),$$
which completes the proof.
\end{proof}

\subsection{A trace formula}

The starting point for our proof of Theorem \ref{QuadThm} is the following trace formula. Combining Corollary \ref{alphaLem} with \eqref{SymmEllTraceForm}, we get
\begin{align*}
    \left\langle \Tr\left(\Theta^n_{E_D}\right) \right\rangle_{\Hh^{\pm}_{N}(\Delta)} = -\frac{q^{-n/2}}{|\Hh^{\pm}_{N}(\Delta)|} \sum_{d|n} \frac{n}{d} \sum_{\deg(P)=\frac{n}{d}} a^*_{1,P^d}  \sum_{D\in\Hh^{\pm}_{N}(\Delta)} \left(\frac{D}{P}\right)^d.
\end{align*}
For convenience, we define
\begin{align*}
    MT^{\pm}(n,N) := - q^{-n/2} \sum_{\substack{d|n \\ 2|d}} \frac{n}{d} \sum_{\deg(P)=\frac{n}{d}}  a^*_{1,P^d} \frac{|\Hh^{\pm}_{N}(P\Delta)|}{|\Hh^{\pm}_{N}(\Delta)|}
\end{align*}
and
\begin{align*}
    ET^{\pm}(n,N) := -\frac{ q^{-n/2}}{|\Hh^{\pm}_{N}(\Delta)|} \sum_{\substack{d|n \\ 2\nmid d}} \frac{n}{d} \sum_{\deg(P)=\frac{n}{d}} a^*_{1,P^d} \sum_{D\in\Hh^{\pm}_{N}(\Delta)} \left(\frac{D}{P}\right).
\end{align*}

\subsection{Estimating $M^{\pm}(n,N)$}

We see that in order to compute $MT^{\pm}(n,N)$ it is enough to prove the following proposition.

\begin{prop}\label{CoprimeProp}
For any prime $P$, we have
$$\frac{|\Hh^{\pm}_{N}(P\Delta )|}{|\Hh^{\pm}_{N}(\Delta)|}  =  \begin{cases} \frac{|P|}{|P|+1} + O(q^{-N/2}) & P\nmid \Delta, \\ 1 & P|\Delta. \end{cases}$$
\end{prop}

We first note that the case where $P|\Delta$ is trivial as in this case $\Hh_{N}^{\pm}(P\Delta) = \Hh_N^{\pm}(\Delta)$. The proof of the remaining part of the proposition, i.e.\ the case where $P\nmid \Delta$, follows immediately from the following two lemmas.

\begin{lem}\label{pmSizeLem}
Let $E$ be an elliptic curve defined over $\Ff_q[T]$. If $M\neq1$, then, for any $\Delta\in\Ff_q[T]$ (not necessarily the discriminant of $E$), we have
$$|\Hh^{\pm}_N(\Delta)| = \frac{1}{2}|\Hh_N(\Delta)| + O_\Delta(q^{N/2}). $$
Moreover, if $M=1$, then either $\Hh^{+}_N(\Delta)=\Hh_N(\Delta)$ or $\Hh^{-}_N(\Delta)=\Hh_N(\Delta)$.
\end{lem}

\begin{proof}
The second part of the lemma follows immediately from the formula for the root number $\epsilon(E_D)$. For the first part, we have
\begin{align*}
    |\Hh^{\pm}_N(\Delta)| & = \sum_{D\in\Hh_N(\Delta)} \frac{1}{2} \big(1 \pm \epsilon_N\epsilon(E)\chi_D(M)\big) = \frac{1}{2}|\Hh_N(\Delta)| \pm \frac{\epsilon_N\epsilon(E)}{2} \sum_{D\in \Hh_N(\Delta)} \chi_D(M).
\end{align*}
Now, by quadratic reciprocity, we have that 
$$\chi_D(M) = (-1)^{\frac{q-1}{2}\deg(M)N} \chi_M(D).$$
Furthermore, we observe that 
\begin{align*}
    \mathcal{G}_{\Delta}(u,\chi_M) & := \sum_{(D,\Delta)=1} \mu^2(D) \chi_M(D) u^{\deg(D)}  = \prod_{P\nmid \Delta} \big(1 + \chi_M(P)u^{\deg(P)}\big) \\
    & = \prod_{P|\Delta} \big(1 + \chi_M(P)u^{\deg(P)}\big)^{-1} \prod_{P|M} \big(1-u^{2\deg(P)}\big)^{-1} \frac{L(u,\chi_M)}{\zeta_q(u^2)}.
\end{align*}
Hence, the above generating series can be analytically extended to the region $|u|\leq q^{-1/2}$ and we conclude that
$$\sum_{D\in \Hh_N(\Delta)} \chi_D(M) = \frac{1}{2\pi i} \oint_{\Gamma} \frac{\mathcal{G}_{\Delta}(u,\chi_M)}{u^{N+1}}\,du \ll q^{N/2} \max_{u\in \Gamma}  \left| \mathcal{G}_{\Delta}(u,\chi_M) \right|, $$
where $\Gamma = \{u : |u|=q^{-1/2}\}$. We also note that 
$$\max_{u\in \Gamma} \Bigg|\prod_{P|\Delta} \big(1 + \chi_M(P)u^{\deg(P)}\big)^{-1} \prod_{P|M} \big(1-u^{2\deg(P)}\big)^{-1}\Bigg| = O_{\Delta}(1).$$

Finally, we use the fact that the Riemann Hypothesis is known for the $L$-function $L(u,\chi_M)$ to get that there exists a unitary matrix $\Theta_M$ of size $\mathcal{M}\times\mathcal{M}$, where $\mathcal M\leq\deg(M)-1$, such that $L(u,\chi_M) = \det(1-\sqrt{q}u\Theta_M)$. Thus
$$\max_{u\in \Gamma} \left|L(u,\chi_M)\right| = \max_{|u|=1} \det(1-u\Theta_M) \ll 2^{\deg(M)} = O(1),$$
and the result follows.
\end{proof}

Next, we estimate the size of $\Hh_{N}(\Delta)$.

\begin{lem}\label{SizeLem}
For any $\Delta\in\Ff_q[T]$, we have
$$|\Hh_{N}(\Delta)| = q^{N-1}(q-1) \prod_{Q|\Delta} \frac{|Q|}{|Q|+1} + O_\Delta(1),$$
where the product is over all prime divisors of $\Delta$.
\end{lem}

\begin{proof}
For any $\Delta\in\Ff_q[T]$, let
$$\mathcal{G}_{\Delta}(u) := \sum_{(D,\Delta)=1} \mu^2(D) u^{\deg(D)} = \sum_{N=0}^{\infty} |\Hh_N(\Delta)|u^N.$$
We can then write $\mathcal{G}_{\Delta}(u)$ as an Euler product:
$$\mathcal{G}_{\Delta}(u) = \prod_{Q\nmid \Delta} \big(1+u^{\deg(Q)}\big) = \prod_{Q|\Delta} \big(1+u^{\deg(Q)}\big)^{-1} \frac{\zeta_q(u)}{\zeta_q(u^2)}. $$
Hence, we get that $\mathcal{G}_{\Delta}(u)$ can be meromorphically extended to the region $|u|<1$ with a simple pole at $u=q^{-1}$. Therefore, if $\Gamma = \{u:|u|=\frac12\}$, then
\begin{align*}
|\Hh_N(\Delta)| & = -\Res_{u=q^{-1}}\left(\frac{\mathcal{G}_{\Delta}(u)}{u^{N+1}}\right) + \frac{1}{2\pi i} \oint_{\Gamma} \frac{\mathcal{G}_{\Delta}(u)}{u^{N+1}}\,du \\
& = q^{N-1}(q-1) \prod_{Q|\Delta} \frac{|Q|}{|Q|+1} + O_\Delta(1),
\end{align*}
which is the desired result.
\end{proof}

\begin{rem}
Note that the error terms in Lemmas \ref{pmSizeLem} and \ref{SizeLem} only depend on the number of prime divisors of $\Delta$. Thus, it follows that the error term in Proposition \ref{CoprimeProp} can be made independent of the prime $P$.
\end{rem}

Using Proposition \ref{CoprimeProp} and Lemma \ref{NiceTraceFormula}, together with the fact that
$$\frac{|P|}{|P|+1} = 1 - \frac{1}{|P|+1},$$
we immediately get that
\begin{align*}
    & MT^{\pm}(n,N) := - q^{-n/2} \sum_{\substack{d|n \\ 2|d}} \frac{n}{d} \sum_{\deg(P)=\frac{n}{d}}  a^*_{1,P^d} \frac{|\Hh^{\pm}_{N}(P\Delta)|}{|\Hh^{\pm}_{N}(\Delta)|} \\ 
    & \hspace{20pt} =  - q^{-n/2} \sum_{\substack{d|n \\ 2|d}} \frac{n}{d} \sum_{\deg(P)=\frac{n}{d}}  a^*_{1,P^d} + q^{-n/2} \sum_{\substack{d|n \\ 2|d}} \frac{n}{d} \sum_{\substack{\deg(P)=\frac{n}{d} \\ P\nmid\Delta }}  a^*_{1,P^d}\left(\frac{1}{|P|+1} + O\left(\frac{1}{q^{N/2}}\right)\right)\\
    & \hspace{20pt} = \eta_2(n) \left(1 + \frac{\Tr\big(\Theta^{n/2}_{sym^2E}\big)}{q^{n/4}} + \frac{\mathcal{D}(n)}{q^{n/2}}\right) + O\left(\frac{1}{q^{N/2}}\right),
\end{align*}
where 
\begin{multline}\label{DDef}
\mathcal{D}(n) := \sum_{\substack{d|n \\ 2|d}} \frac{n}{d} \sum_{\substack{\deg(P)=\frac{n}{d} \\ P\nmid \Delta }}  \frac{a^*_{1,P^d}}{|P|+1} \\
- \sum_{\substack{d|n \\ 2|d}} \frac{n}{d} \sum_{\substack{\deg(P)=\frac{n}{d} \\ P|\Delta }} \big(a^*_{1,P^d} - a^*_{2,P^{d/2}} +1\big) \ll \tau(n) + \deg(\Delta).
\end{multline}

\subsection{Bounding $ET^{\pm}(n,N)$}

To bound $ET^{\pm}(n,N)$, we refer to the work of Comeau-Lapointe \cite{C-L}. To align with the notation from \cite{C-L}, we define
$$\Hh_{N,C} = \{D \in \Hh_{N}(\Delta) : D\equiv C \bmod{N_E} \},$$
where $N_E$ is the conductor of the elliptic curve $E$ as defined in, e.g., \cite[Lecture 1]{Ulmer} (see also \cite[Section 2.1]{C-L}). In particular, with our notation, we have $\mathfrak{n}= \mathfrak{n}_E = \deg(N_E)-4$. 

Next, we define
$$S_C(n,N):= -\frac{n}{q^{n/2}|\Hh_{N,C}|} \sum_{\deg(P)=n} a^*_{1,P} \sum_{D\in\Hh_{N,C}} \left(\frac{D}{P}\right).$$
Then, \cite[Proposition 7.2]{C-L} shows that for any $\epsilon>0$, $N> 4\mathfrak{n}+16$ and $C$ coprime to $N_E$, we have
$$S_C(n,N) \ll_{\epsilon} (n+N)N^{2\mathfrak{n}+11}\left(\frac{1}{q^{N/8}} + \frac{1}{q^{\epsilon N}} + \frac{q^{n/2}}{q^{(1-\epsilon)N}}\right).$$
Moreover, as we saw in the proof of Lemma \ref{pmSizeLem}, as long as $P\not=M$,
\begin{align}\label{P!=M}
    \sum_{D\in\Hh^{\pm}_N(\Delta)} \left(\frac{D}{P}\right) = \frac{1}{2} \sum_{D\in\Hh_N(\Delta)} \left(1 \pm \epsilon_N \epsilon(E) \left(\frac{D}{M}\right)\right)\left(\frac{D}{P}\right)\ll q^{N/2},
\end{align}
and so as long as $M$ is not a prime of degree dividing $n$,
\begin{align*}
    ET^{\pm}(n,N) & = -\frac{ q^{-n/2}}{|\Hh^{\pm}_{N}(\Delta)|} \sum_{\substack{d|n \\ 2\nmid d}} \frac{n}{d} \sum_{\deg(P)=\frac{n}{d}} a^*_{1,P^d} \sum_{D\in\Hh^{\pm}_{N}(\Delta)} \left(\frac{D}{P}\right)\\
    & = \sum_{\substack{C \bmod{N_E} \\ (C,N_E)=1\\ \chi_C(M)=\pm\epsilon_N\epsilon(E) }} \frac{|\Hh_{N,C}|}{|\Hh_N^{\pm}(\Delta)|} S_C(n,N) + O\Bigg(\frac{ 1}{q^{(n+N)/2}} \sum_{\substack{d|n \\d\geq 3}} \frac{n}{d} \sum_{\deg(P)=\frac{n}{d}} |a^*_{1,P^d}|\Bigg)\\ 
    & = O_{\epsilon}\left((n+N)N^{2\mathfrak{n}+11}\left(\frac{1}{q^{N/8}} + \frac{1}{q^{\epsilon N}} + \frac{q^{n/2}}{q^{(1-\epsilon)N}}\right) + \frac{1}{q^{N/2+n/6}}\right).
\end{align*}

Finally, in the case that $M=P$, then by using \eqref{P!=M}, we see that 
\begin{align}\label{M=P}
    \sum_{D\in\Hh^{\pm}_N(\Delta)} \left(\frac{D}{P}\right) =  \pm\epsilon_N\epsilon(E) |\Hh^{\pm}_N(\Delta)| + O\big(q^{N/2}\big).
\end{align}
Therefore, if additionally $\deg(P)=\frac{n}{d}$ for some odd $d|n$, then this prime would contribute a term  
$$\mp\epsilon_N\epsilon(E) \frac{n}{d}\frac{ a^*_{1,P^d}}{q^{n/2}}$$
to $ET^{\pm}(n,N)$, which we could incorporate into the term $\mathcal{D}(n)q^{-n/2}$ from the previous subsection.

\subsection{Proof of Theorem \ref{QuadThm}}

We are now in position to complete the proof of Theorem \ref{QuadThm}.

\begin{proof}[Proof of Theorem \ref{QuadThm}]
Combining the results from the previous subsections, for any $\epsilon>0$ and $N> 4\mathfrak{n}+16$, we have 
\begin{align*}
    \left\langle \Tr\left(\Theta^n_{E_D}\right) \right\rangle_{\Hh^{\pm}_{N}(\Delta)} & = MT^{\pm}(n,N) +ET^{\pm}(n,N) \\
    & = \eta_2(n) \left(1 + \frac{\Tr\big(\Theta^{n/2}_{sym^2E}\big)}{q^{n/4}} +\frac{\mathcal{D}(n)}{q^{n/2}}\right) \\
    & + O_{\epsilon}\left(  (n+N)N^{2\mathfrak{n}+11}\left(\frac{1}{q^{N/8}} + \frac{1}{q^{\epsilon N}} + \frac{q^{n/2}}{q^{(1-\epsilon)N}}\right) + \frac{1}{q^{N/2+n/6}} \right).
\end{align*}
We may then absorb the term $q^{-N/2-n/6}$ into the other error terms which gives the desired result.
\end{proof}

\subsection{Proof of Corollary \ref{QuadCor}}\label{Subsection:QuadCor}

Recall that we have, for any unitary $N\times N$ matrix $U$,
\begin{align*}
\mathcal{D}(U,f) :=   \sum_{j=1}^{N} \sum_{n\in\Z} f\left(N\left(\frac{\theta_j}{2\pi} -n\right)\right) =  \frac{1}{N} \sum_{n\in\mathbb{Z}} \widehat{f}\left( \frac{n}{N} \right) \Tr(U^n),
\end{align*}
where the $\theta_j$ run over the eigenangles of the matrix $U$. In particular, since we know that 
\begin{align*}
\int_{\mathrm{O}(N)} \Tr(U^n)\,dU = 
\begin{cases} N & \text{if $n=0$}, \\ 
\eta_2(n) & \text{if $n\not=0$}, \end{cases}   
\end{align*} 
we get that
\begin{align*}
   \int_{\mathrm{O}(\mathfrak{n} + 2N)} \mathcal{D}(U,f)\,dU & =  \frac{1}{\mathfrak{n}+2N} \sum_{n\in\mathbb{Z}} \widehat{f}\left( \frac{n}{\mathfrak{n}+2N} \right) \int_{\mathrm{O}(\mathfrak{n}+2N)}\Tr(U^n)\,dU \\
   & = \widehat{f}(0) + \frac{2}{\mathfrak{n}+2N} \sum_{n=1}^{\infty} \widehat{f}\left(\frac{2n}{\mathfrak{n}+2N}\right).
\end{align*}

Now, using Theorem \ref{QuadThm} to average over the quadratic twist family, we find that if $\supp(\widehat{f}) \subset (-1+\delta,1-\delta)$ for some $\delta>0$, then
\begin{align}\label{1-level-firstexpression}
    \big\langle\mathcal{D}(\Theta_{E_D},f) \big\rangle_{\Hh_N^{\pm}(\Delta)} & =  \frac{1}{\mathfrak{n}_{E_D} } \sum_{n\in\mathbb{Z}} \widehat{f}\left( \frac{n}{\mathfrak{n}_{E_D}} \right) \big\langle \Tr(\Theta_{E_D}^n) \big\rangle_{\Hh_N^{\pm}(\Delta)} \nonumber\\
    &\hspace{-50pt} = \widehat{f}(0) + \frac{2}{\mathfrak{n}+2N}\sum_{n=1}^{(1-\delta)(\frac{\mathfrak{n}}{2}+N)}\widehat{f}\left( \frac{2n}{\mathfrak{n}+2N} \right)\left(1 + \frac{\Tr\big(\Theta^{n}_{sym^2E}\big)}{q^{n/2}} +\frac{\mathcal{D}(2n)}{q^{n}}\right) \\
    &\hspace{-50pt} + O_{\epsilon}\left( \sum_{n=1}^{(1-\delta)(\mathfrak{n}+2N)}  (n+N)N^{2\mathfrak{n}+10} \left(\frac{1}{q^{N/8}} + \frac{1}{q^{\epsilon N}} + \frac{q^{n/2}}{q^{(1-\epsilon)N}}\right)\right).\nonumber
\end{align}
We divide the right-hand side above into pieces that we analyze separately. 

Firstly, we note that the error term in \eqref{1-level-firstexpression} is bounded by
$$O_\epsilon\left(N^{2\mathfrak{n}+12}\left(\frac{1}{q^{N/8}} + \frac{1}{q^{\epsilon N}} + q^{(\epsilon-\delta)N}\right) \right) = O_\epsilon\left(\frac{1}{q^{\epsilon'N}}\right)$$
for some $\epsilon'>0$ as long as $\epsilon<\delta$. Next, we use the assumption that $\supp(\widehat{f})\subset (-1+\delta,1-\delta)$ to extend the sum to be over all positive $n$. Hence, we can write the main term in \eqref{1-level-firstexpression} as
\begin{multline}\label{O-integral-repr}
    \widehat{f}(0) +  \frac{2}{\mathfrak{n}+2N}\sum_{n=1}^{(1-\delta)(\frac{\mathfrak{n}}{2}+N)}\widehat{f}\left( \frac{2n}{\mathfrak{n}+2N} \right)\\
    = \widehat{f}(0) + \frac{2}{\mathfrak{n}+2N} \sum_{n=1}^{\infty} \widehat{f}\left(\frac{2n}{\mathfrak{n}+2N}\right) 
    = \int_{\mathrm{O}(\mathfrak{n}+2N)} \mathcal{D}(U,f)\,dU .
\end{multline}

For the secondary terms in \eqref{1-level-firstexpression}, we split the sum over $n$ into two parts. Let $\phi(N)$ be any function (to be determined later). Then, we use the fact that $f$ is a Schwartz function to get that $\widehat{f}(x+y) = \widehat{f}(x) + O(y)$ and so
\begin{align}\label{secondary-firstexpression}
   \sum_{n=1}^{\phi(N)}\widehat{f}\left( \frac{2n}{\mathfrak{n}+2N} \right) \frac{\Tr\big(\Theta^{n}_{sym^2E}\big)}{q^{n/2}} & = \sum_{n=1}^{\phi(N)}\left(\widehat{f}(0) + O\left(\frac{n}{N}\right)\right)\frac{\Tr\big(\Theta^{n}_{sym^2E}\big)}{q^{n/2}}\nonumber\\
   & = \widehat{f}(0) \sum_{n=1}^{\phi(N)} \frac{\Tr\big(\Theta^n_{sym^2E}\big)}{q^{n/2}} + O\left(\frac{\phi(N)}{N}\right).
\end{align}
Next, combining \eqref{amPkDef} and \eqref{SymmEllTraceForm}, we get that
$$\frac{L'(u,sym^2E)}{L(u,sym^2E)} = -\frac{1}{u}\sum_{n=1}^{\infty} \Tr\left(\Theta_{sym^2E}^n\right)\left(\sqrt{q}u\right)^n.$$
Therefore, extending the sum in \eqref{secondary-firstexpression} to be over all positive $n$, while gaining an additional error term of order $q^{-\phi(N)/2}$, we find that the main contribution from \eqref{secondary-firstexpression} equals
$$-\frac{\widehat{f}(0)}{q} \frac{L'(q^{-1},sym^2E)}{L(q^{-1},sym^2E)}.$$
For the remaining terms with $n>\phi(N)$, we use the fact that $\widehat{f}$ is bounded to get 
$$\sum_{\phi(N)<n\leq(1-\delta)(\frac{\mathfrak{n}}{2}+N)}\widehat{f}\left( \frac{2n}{\mathfrak{n}+2N} \right) \frac{\Tr\big(\Theta^{n}_{sym^2E}\big)}{q^{n/2}} \ll q^{-\phi(N)/2}.$$
Combining the above observations, we set $\phi(N)=N^{\epsilon}$ and conclude that 
\begin{align*}
    \frac{2}{\mathfrak{n}+2N}\sum_{n=1}^{(1-\delta)(\frac{\mathfrak{n}}{2}+N)}\widehat{f}\left( \frac{2n}{\mathfrak{n}+2N} \right) \frac{\Tr\big(\Theta^{n}_{sym^2E}\big)}{q^{n/2}}
    = -\frac{\widehat{f}(0)}{N}\frac{L'(q^{-1},sym^2E)}{qL(q^{-1},sym^2E)} + O_{\epsilon}\left(\frac{1}{N^{2-\epsilon}}\right).
\end{align*}

Finally, we consider also the remaining secondary term in \eqref{1-level-firstexpression}. Similarly as in the treatment of the first secondary term above, we get that
\begin{align*}
    &\frac{2}{\mathfrak{n}+2N} \sum_{n=1}^{(1-\delta)(\frac{\mathfrak{n}}{2}+N)} \widehat{f}\left(\frac{2n}{\mathfrak{n}+2N}\right) \frac{\mathcal{D}(2n)}{q^n} \\
    & = \frac{\widehat{f}(0)}{N} \sum_{n=1}^{\infty}\sum_{\substack{d|2n\\ 2|d }} \Bigg( \sum_{\substack{\deg(P)=\frac{2n}{d}\\ P\nmid \Delta}} \frac{\deg(P)a^*_{1,P^{d}}}{q^n(|P|+1)} \\ 
    & \hspace{116pt}- \sum_{\substack{\deg(P) = \frac{2n}{d} \\ P|\Delta }} \frac{\deg(P) \big(a^*_{1,P^{d}}-a^*_{2,P^{d/2}}+1\big)}{q^n}  \Bigg) + O_{\epsilon}\left(\frac{1}{N^{2-\epsilon}}\right)\\
    & = \frac{\widehat{f}(0)}{N}\Bigg(\sum_{P \nmid \Delta} \frac{\deg(P)}{|P|+1} \sum_{d=1}^{\infty}\frac{a^*_{1,P^{2d}}}{|P|^d} - \sum_{P|\Delta} \deg(P) \sum_{d=1}^{\infty} \frac{a^*_{1,P^{2d}} - a^*_{2,P^d}+1}{|P|^d} \Bigg)+ O_{\epsilon}\left(\frac{1}{N^{2-\epsilon}}\right).
\end{align*}
This concludes the proof of Corollary \ref{QuadCor}.

\section{Cubic twists}\label{Section:cubic}

In this section, we prove Theorems \ref{CubicOLD} and \ref{CubicLowTerm}.

\subsection{A formula for $a_P(\widetilde{E})$}\label{Subsection:cubiccoeff}

Recall that we denote by $\widetilde{E}$ an elliptic curve given by the equation
\begin{align}\label{E-tilde-Eq}
    \widetilde{E}: y^2 = x^3+B,
\end{align}
where $B\in\Ff_q[T]$, and that for every $D\in\mathcal{F}_{N}(B)$, we consider the cubic twist
$$\widetilde{E}_D: y^2=x^3+BD^2.$$
For any elliptic curve $\widetilde{E}$ of the above form, and any prime $P$, we define
$$\lambda_P = \lambda_P(\widetilde{E}) := \frac{1}{q^{\deg(P)/2}}\sum_{F\bmod P} \left(\frac{F^2-B}{P}\right)_3,$$
where $\left(\frac{\cdot}{P}\right)_3$ is the cubic residue symbol on $\Ff_q[T]/(P) \cong \Ff_{q^{\deg(P)}}$. Note in particular that the Weil bound implies that $|\lambda_P| \leq 1$.

Similar to the case of quadratic twists, we will use the fact that $a_P(\widetilde{E}_D)$ can be expressed in terms of the number of points of $\widetilde{E}_D(P)$ (as a character sum) and then use this information to understand how these coefficients change as we vary $D$.

\begin{lem}\label{Lem:cubictwistcoeff}
For any elliptic curve $\widetilde{E}$ given by an equation of the form \eqref{E-tilde-Eq}, prime $P$ and $D\in\mathcal{F}_{N}(B)$, we have
\begin{align}\label{aPEtilde}
    a_P(\widetilde{E}) = -(\lambda_P + \overline{\lambda}_P)
\end{align}
and 
\begin{align}\label{LambdaDForm}
    \lambda_P(\widetilde{E}_D) = \left(\frac{D}{P}\right)_3^2 \lambda_P(\widetilde{E}).
\end{align}
\end{lem}

\begin{proof}
Recall that 
$$\#\widetilde{E}(P) = q^{\deg(P)}+1 -a_P(\widetilde{E}) q^{\deg(P)/2}.$$
We also know from our discussion in Section \ref{Section:quadratic} that there will always be exactly one point lying above $\infty$ on $\widetilde{E}(P)$. Moreover, since our curves have the form \eqref{E-tilde-Eq}, we observe that for a prime $P$ and a finite point  $F\in\Ff_q[T]/(P)$, the number of points lying over $F$ on $\widetilde{E}(P)$ is
$$\begin{cases}
3 & \text{if $F^2-B$ is a non-zero perfect cube mod $P$,}\\
1 & \text{if $F^2-B\equiv 0$ mod $P$,}\\
0 & \text{if $F^2-B$ is not a perfect cube mod $P$.}
\end{cases}$$
We capture this information in the character sum
\begin{align*}
    \#\widetilde{E}(P) & =1+ \sum_{F \bmod{P}}\left(1 + \left(\frac{F^2-B}{P}\right)_3 + \left(\frac{F^2-B}{P}\right)^2_3\right) \\
    &= q^{\deg(P)}+1 + \sum_{F \bmod{P}}\left( \left(\frac{F^2-B}{P}\right)_3 + \left(\frac{F^2-B}{P}\right)^2_3\right)\\
    & = q^{\deg(P)}+1 + (\lambda_P+\overline{\lambda}_P)q^{\deg(P)/2},
\end{align*}
which proves \eqref{aPEtilde}.

To prove \eqref{LambdaDForm}, we first consider the case when $P\nmid D$. Then $D$ is invertible modulo $P$ and so
\begin{align*}
    \lambda_P(\widetilde{E}_D) &= \frac{1}{q^{\deg(P)/2}} \sum_{F\bmod{P}}\left(\frac{F^2-BD^2}{P}\right)_3 \\
    & = \left(\frac{D^2}{P}\right)_3 \frac{1}{q^{\deg(P)/2}} \sum_{F\bmod{P}}\left(\frac{(FD^{-1})^2-B}{P}\right)_3\\
    & = \left(\frac{D}{P}\right)^2_3 \lambda_P(\widetilde{E}),
\end{align*}
where the last equality comes from the fact that as $F$ runs over all the elements mod $P$ so does $FD^{-1}$. Finally, if $P|D$, then we get 
\begin{align*}
    \lambda_P(\widetilde{E}_D) &= \frac{1}{q^{\deg(P)/2}} \sum_{F\bmod{P}}\left(\frac{F}{P}\right)^2_3=0,
\end{align*}
which clearly equals $\left(\frac{D}{P}\right)^2_3 \lambda_P(\widetilde{E})$ in this case.
\end{proof}

It is tempting to try to conclude from \eqref{aPEtilde} that $\lambda_P=-\alpha_P$. This is true if and only if $|\lambda_P|=1$. However, we will see that the expected value of $|\lambda_P|^2$, for primes $P$ of large degree, is $\frac{1}{2}$ (cf.\ Corollary \ref{maincor}) and so $\lambda_P$ is in general not equal to $-\alpha_P$. Therefore, it is not necessarily true that $\alpha_P(\widetilde{E}_D) = \left(\frac{D}{P}\right)_3^2 \alpha_P(\widetilde{E})$. This causes some minor issues when we calculate the expected values of traces of the Frobenius, as we need to first write everything in terms of $\lambda_P$ instead of the more natural $\alpha_P$. The following lemma presents the essential parts of this reformulation.

\begin{lem}\label{atolambda}
If $P$ is a prime of good reduction for $\widetilde E$, then we have the following:
\begin{enumerate}
    \item $a^*_{1,P} = -(\lambda_P+\overline{\lambda}_P)$,
    \item $a^*_{1,P^2} = \lambda_P^2+\overline{\lambda}_P^2 + 2(|\lambda_P|^2-1)$,
    \item $a^*_{1,P^3} = -(\lambda_P^3 + \overline{\lambda}_P^3) - 3(|\lambda_P|^2-1)(\lambda_P+\overline{\lambda}_P)$.
\end{enumerate}
\end{lem}

\begin{proof}
For $(1)$, we have that
$$a^*_{1,P} = a_{P} = -(\lambda_P+\overline{\lambda}_P).$$
To prove $(2)$, we note that
\begin{align*}
    a^*_{1,P^2} & = \alpha_P^2+\beta_P^2 = (\alpha_P+\beta_P)^2- 2 
    = (\lambda_P+\overline{\lambda}_P)^2-2  = \lambda_P^2+\overline{\lambda}_P^2 + 2(|\lambda_P|^2-1).
\end{align*}
Finally, for $(3)$, we have that
\begin{align*}
    a^*_{1,P^3} & = \alpha_P^3+\beta_P^3 = (\alpha_P+\beta_P)^3-3(\alpha_P+\beta_P) = -(\lambda_P+\overline{\lambda}_P)^3+3(\lambda_P+\overline{\lambda}_P) \\
    & = -(\lambda^3_P+\overline{\lambda}^3_P)-3(|\lambda_P|^2-1)(\lambda_P+\overline{\lambda}_P),
\end{align*}
which concludes the proof.
\end{proof}

\subsection{Trace formulas}

Applying Lemma \ref{TruncTrForm} with $m=2$ and using Lemmas \ref{RHbound} and \ref{atolambda}, we obtain
\begin{multline*}
    -q^{n/2}\Tr(\Theta_{\widetilde{E}}^n) = n\sum_{\deg(P)=n}a^*_{1,P} + \frac{n}{2}\sum_{\deg(P)=\frac{n}{2}} a^*_{1,P^2} +  O\Bigg(q^{n/6} +  n\sum_{\substack{d|n \\ d>2}} \deg(\Delta_{n/d})\Bigg) \\
     = -\eta_2(n)q^{n/2} -n \sum_{\deg(P)=n} (\lambda_P + \overline{\lambda}_P) +O\bigg(\eta_2(n)nq^{n/4}+q^{n/6} +  n\sum_{d|n} \deg(\Delta_{n/d})\bigg).
\end{multline*}
Therefore, if for every prime $P$, we define
\begin{align}\label{EP}
    E_P = E_P(N) := \frac{1}{|\F_N(B)|}\sum_{D\in\F_N(B)} \left(\frac{D}{P}\right)_3,
\end{align}
then we get
\begin{multline}\label{CubicTrForm1}
    \big\langle\Tr(\Theta_{\widetilde{E}_D}^n)\big\rangle_{\F_N(B)} = \eta_2(n) + \frac{n}{q^{n/2}} \sum_{\deg(P)=n} \left(\lambda_P\overline{E}_P +\overline{\lambda}_PE_P\right) \\ + O\left(\frac{\eta_2(n)n}{q^{n/4}}+\frac{1}{q^{n/3}} + \frac{n(\deg(\Delta) + \tau(n))}{q^{n/2}} \right),
\end{multline}
where we have used Lemma \ref{BadPrimeBound} to bound the contribution from the primes of bad reduction dividing elements in $\F_N(B)$. 

On the other hand, if we apply Lemma \ref{TruncTrForm} with $m=3$ instead of $m=2$, together with Lemma \ref{atolambda}, then we get a different trace formula:
\begin{multline*}
    -q^{n/2}\Tr(\Theta_{\widetilde{E}}^n) =  -n\sum_{\deg(P)=n} \left(\lambda_P+\overline{\lambda}_P\right) + \frac{n}{2} \sum_{\deg(P)=\frac{n}{2}} \left(\lambda_P^2+\overline{\lambda}_P^2+2(|\lambda_P|^2-1)\right) \\
    - \frac{n}{3} \sum_{\deg(P)=\frac{n}{3}} \left( \lambda_P^3+\overline{\lambda}_P^3 +3(|\lambda_P|^2-1)(\lambda_P+\overline{\lambda}_P) \right) + O\bigg(q^{n/8} + n\sum_{d|n} \deg(\Delta_{n/d}) \bigg).
\end{multline*}
Now, we note that as long as $P\nmid D$, we have $|\lambda_P(\widetilde{E}_D)|^2 = |\lambda_P(\widetilde{E})|^2$ and $\lambda^3_P(\widetilde{E}_D) = \lambda^3_P(\widetilde{E})$. Therefore, if we define
\begin{align}
    M(n,N) := -\frac{n/2}{q^{n/2}}  \sum_{\deg(P)=\frac{n}{2} } 2(|\lambda_P|^2-1)  \frac{|\F_N(PB)|}{|\F_N(B)|}, 
\end{align}
\begin{align}
    S(n,N)  := \frac{n/3}{q^{n/2}}  \sum_{ \deg(P)=\frac{n}{3}} \big(\lambda^3_P + \overline{\lambda}^3_P\big) \frac{|\F_N(PB)|}{|\F_N(B)|}, 
\end{align}
\begin{multline}
    E(n,N):= \frac{n}{q^{n/2}} \sum_{\deg(P)=n} \big(\lambda_P\overline{E}_P+\overline{\lambda}_PE_P\big)  \\ - \frac{n/2}{q^{n/2}} \sum_{\deg(P)=\frac{n}{2}} \ \big(\lambda_P^2E_P+\overline{\lambda}^2_P\overline{E}_P\big)  + 
    \frac{n/3}{q^{n/2}} \sum_{\deg(P)=\frac{n}{3}} \ 3(|\lambda_P|^2-1)\big(\lambda_P\overline{E}_P+\overline{\lambda}_PE_P\big),  
\end{multline}
with each $\lambda_P = \lambda_P(\widetilde{E})$, then we have
\begin{align}\label{CubicTrForm2}
    \big\langle\Tr(\Theta_{\widetilde{E}_D}^n)\big\rangle_{\F_N(B)} = M(n,N) + S(n,N) + E(n,N) + O\left(\frac{1}{q^{3n/8}} + \frac{n(\deg(\Delta)+\tau(n))}{q^{n/2}}\right),
\end{align}
where we use the same bounds as in \eqref{CubicTrForm1}, together with Remark \ref{Sec2Rem}, to handle the primes of bad reduction.

Hence, in order to prove Theorems \ref{CubicOLD} and \ref{CubicLowTerm}, we need to bound $E_P$ and compute $\frac{|\F_N(PB)|}{|\F_N(B)|}$, which we can view as the probability that a random $D\in\F_N(B)$ is coprime to $P$.

\subsection{Bounding $E_P(N)$}

In this short subsection, we prove the following bound on $E_P(N)$.

\begin{prop}\label{EPprop}
For any $\epsilon>0$ and any prime $P$, we have
$$E_P(N) \ll_{\epsilon} \frac{e^{2\deg(P)}}{Nq^{(\frac{1}{2}-\epsilon)N}}.$$
\end{prop}

\begin{proof}
We denote the cubic residue symbol modulo $P$ by
$$\psi_P = \left(\frac{\cdot}{P}\right)_3$$
and consider the generating series
$$\mathcal{G}(u,\psi_P) := \sum_{N=0}^{\infty} \sum_{D\in\F_N(B)} \psi_P(D) u^N.$$
Every $D\in\F_N(B)$ can be written as $D=D_1D_2^2$, where $\deg(D_1D_2)=N$, $3|\deg(D_1D_2^2)$ and $D_1,D_2$ are monic, square-free, coprime to each other and coprime to $B$. Using this, we obtain that 
\begin{align*}
    \mathcal{G}(u,\psi_P) & = \frac{1}{3} \sum_{\substack{D_1,D_2 \\ (D_1D_2,B)=1 \\  }} \mu^2(D_1D_2)\psi_P(D_1D_2^2)\left(1+\xi_3^{\deg(D_1D_2^2)} + \xi_3^{2\deg(D_1D_2^2)} \right)u^{\deg(D_1D_2)} \\
    & = \frac{1}{3} \big(H_0(u,\psi_P) + H_1(u,\psi_P) + H_2(u,\psi_P) \big),
\end{align*}
where $\xi_3$ is a primitive cube root of unity and 
\begin{align*}
    H_j(u,\psi_P) := \sum_{\substack{D_1,D_2 \\ (D_1D_2,B)=1 \\  }} \mu^2(D_1D_2)\psi_P(D_1D_2^2)\xi_3^{j\deg(D_1D_2^2)}u^{\deg(D_1D_2)}.
\end{align*}
Writing $H_j(u,\psi_P)$ as a product over primes, we get 
\begin{align*}
    H_j(u,\psi_P) & = \prod_{Q\nmid B} \left(1 + \psi_P(Q)(\xi_3^ju)^{\deg(Q)} + \psi^2_P(Q)(\xi_3^{2j}u)^{\deg(Q)}\right)\\
    & = L(\xi^j_3u,\psi_P)L(\xi_3^{2j}u,\psi^2_P) \widetilde{H}_j(u,\psi_P),
\end{align*}
where $L(u,\psi_P)$ is the $L$-function attached to the Dirichlet character $\psi_P$ and  $\widetilde{H}_j(u,\psi_P)$ is a function that has an Euler product with factors of the form $\left(1+O(u^{2\deg(Q)})\right)$ for all $Q\nmid B$ (respectively, $\left(1+O(u^{\deg(Q)})\right)$ for $Q\mid B$), and so is analytic in the region $|u|<q^{-1/2}$. 

Now, since $\psi_P$ and $\psi_P^2$ are both non-trivial Dirichlet characters modulo $P$, we get that $H_j(u,\psi_P)$ is analytic in the region $|u|<q^{-1/2}$ and hence so is $\mathcal{G}(u,\psi_P)$. Therefore, if $\Gamma := \{u : |u|=q^{-1/2-\epsilon}\}$, then we use \cite[Proposition 1.2]{Lumley}\footnote{Note that Lumley \cite{Lumley} is assuming that $q\equiv 1 \bmod{4}$. However, this assumption is not important for the proof of \cite[Proposition 1.2]{Lumley} and the same result holds also in the case $q\equiv 3 \bmod{4}$.} to get
\begin{align*}
    \sum_{D\in\F_N(B)}\left(\frac{D}{P}\right)_3 = \frac{1}{2\pi i} \oint_{\Gamma} \frac{\mathcal{G}(u,\psi_P)}{u^{N+1}}du  \ll \max_{u\in \Gamma} \left(\frac{|\mathcal{G}(u,\psi_P)|}{|u|^{N}}\right)  \ll_{\epsilon} e^{2\deg(P)}q^{(\frac{1}{2}+\epsilon)N}.
\end{align*}
To conclude, we refer to Corollary \ref{FNBsize} below which implies that $|\F_N(B)| \sim  cNq^{N}$ for some non-zero constant $c$.
\end{proof}

\subsection{Proof of Theorem \ref{CubicOLD}}

Using the results from the previous subsections, we are now ready to complete the proof of Theorem \ref{CubicOLD}.

\begin{proof}[Proof of Theorem \ref{CubicOLD}]
Applying Proposition \ref{EPprop} to \eqref{CubicTrForm1}, we get
\begin{align*}
    \big\langle\Tr(\Theta_{\widetilde{E}_D}^n)\big\rangle_{\F_N(B)} & = \eta_2(n) + \frac{n}{q^{n/2}} \sum_{\deg(P)=n} \left(\lambda_P\overline{E}_P +\overline{\lambda}_PE_P\right) \\
    &\hspace{95pt}+ O\left(\frac{\eta_2(n)n}{q^{n/4}}+\frac{1}{q^{n/3}} + \frac{n(\deg(\Delta) + \tau(n))}{q^{n/2}} \right)\\
    & = \eta_2(n) + O_{\epsilon}\left(\frac{q^{n/2}e^{2n}}{Nq^{(\frac{1}{2}-\epsilon)N}} + \frac{\eta_2(n)n}{q^{n/4}}+\frac{1}{q^{n/3}} + \frac{n(\deg(\Delta) + \tau(n))}{q^{n/2}}\right).
\end{align*}
Hence, if $\supp(\widehat{f}) \subset \left(-\alpha,\alpha \right)$  for some $\alpha<\frac{1}{2} - \frac{2}{4+\log q}$, we use \eqref{Poisson} to get 
\begin{align*}
    \big\langle \mathcal{D}(\Theta_{\widetilde{E}_D},f) \big\rangle_{\F_N(B)} = & \frac{1}{\mathfrak{n}+2N} \sum_{n\in\mathbb{Z}} \widehat{f}\left(\frac{n}{\mathfrak{n}+2N}\right) \big\langle\Tr(\Theta_{\widetilde{E}_D}^n)\big\rangle_{\F_N(B)} \\
    & \hspace{-50pt} = \widehat{f}(0) + \frac{2}{\mathfrak{n}+2N} \sum_{n=1}^{\alpha(\mathfrak{n}+2N)} \widehat{f}\left(\frac{n}{\mathfrak{n}+2N}\right) \big\langle\Tr(\Theta_{\widetilde{E}_D}^n)\big\rangle_{\F_N(B)} \\
     & \hspace{-50pt} = \widehat{f}(0) + \frac{2}{\mathfrak{n}+2N} \sum_{n=1}^{\alpha(\frac{\mathfrak{n}}{2}+N)} \widehat{f}\left(\frac{2n}{\mathfrak{n}+2N}\right)\\
      & \hspace{-50pt} + O_{\epsilon}\left(\frac{1}{N}\sum_{n=1}^{\alpha(\mathfrak{n}+2N)} \left(\frac{q^{n/2}e^{2n}}{Nq^{(\frac{1}{2}-\epsilon)N}} + \frac{\eta_2(n)n}{q^{n/4}}+\frac{1}{q^{n/3}} + \frac{n(\deg(\Delta) + \tau(n))}{q^{n/2}}\right)\right) \\ 
     & \hspace{-50pt} = \int_{\mathrm{O}(\mathfrak{n}+2N)} \mathcal{D}(U,f)\,dU + O\left(\frac{1}{N}\right),
\end{align*}
where we fix a sufficiently small $\epsilon$ and use \eqref{O-integral-repr} in the last step.
\end{proof}

\subsection{The probability of being coprime to $P$}

Recall that for the quadratic twists, we proved in Proposition \ref{CoprimeProp} that
$$\frac{|\Hh^{\pm}_N(P\Delta)|}{|\Hh^{\pm}_N(\Delta)|} = \frac{|P|}{|P|+1} + O\big(q^{-N/2}\big)$$
for all $P\nmid\Delta$. To pick out the lower order terms in Theorem \ref{CubicLowTerm}, we need a similar result for $\F_N(B)$ with an error term that decays as $N$ tends to infinity. Proving such a result is a little more delicate due to the fact that the generating series for $\Hh_N(\Delta)$ has a simple pole at $q^{-1}$, whereas the generating series for $\F_N(B)$ has a double pole at $q^{-1}$. That being said, the rest of this subsection will be devoted to proving the following proposition.

\begin{prop}\label{CubicCoprimeProp}
Let $P$ be a prime of degree $m$. If $P\nmid B$, then
$$\frac{|\F_N(PB)|}{|\F_N(B)|} =1 + \sum_{a=1}^{\lfloor\frac{N}{m}\rfloor} \left(\frac{-2}{q^m}\right)^a\left(1-\frac{am}{N}\right) + O\left(\frac{1}{Nq^m}\right), $$
whereas if $P|B$, then
$$\frac{|\F_N(PB)|}{|\F_N(B)|} =1. $$
\end{prop}

The case where $P|B$ is trivial since in this case $\F_N(B) = \F_N(PB)$. Thus, we will consider only the case $P\nmid B$. Towards this goal, we define the generating series
$$\mathcal{G}(u;B)  := \sum_{N=0}^{\infty} |\F_N(B)|u^N.$$
For any analytic function $K(u)$ defined in an open neighborhood of the origin, we define $[u^d]K(u)$ as the $d$th coefficient in the Taylor expansion of $K(u)$ around $0$. Therefore,
$$\frac{|\F_{N}(PB)|}{|\F_{N}(B)|} = \frac{[u^N]\mathcal{G}(u;PB)}{[u^N]\mathcal{G}(u;B)}.$$

Similar to when we bounded $E_P$ (see Proposition \ref{EPprop}), we note that every element in $\F_N(B)$ can be written as $D_1D_2^2$ where $\deg(D_1D_2)=N$, $3|\deg(D_1D_2^2)$ and the $D_i$ are monic, square-free, coprime to each other and coprime to $B$. Hence, we obtain that
\begin{align*}
    \mathcal{G}(u;B) & = \frac{1}{3} \sum_{\substack{D_1,D_2  \\ (D_1D_2,B)=1 }}  \mu^2(D_1D_2)\left(1+\xi_3^{\deg(D_1D_2^2)}+\xi_3^{2\deg(D_1D_2^2)}\right) u^{\deg(D_1D_2)} \\
    & = \frac{1}{3}\big( H_0(u;B) + H_1(u;B)+H_2(u;B) \big),
\end{align*}
where $\xi_3$ is a primitive cube root of unity and
\begin{align*}
    H_j(u;B) := & \sum_{\substack{D_1,D_2 \\ (D_1D_2,B)=1}} \mu^2(D_1D_2) \xi_3^{j\deg(D_1D_2^2)}u^{\deg(D_1D_2)} \\
    = & \prod_{Q\nmid B} \left(1+(\xi_3^ju)^{\deg(Q)}+(\xi_3^{2j}u)^{\deg(Q)}\right).
\end{align*}
It is clear that $H_1(u;B) = H_2(u;B)$ and so we can write
\begin{align*}
    \mathcal{G}(u;B) = \frac{1}{3} \big(H_0(u;B) + 2H_1(u;B)\big).
\end{align*}

\begin{lem}\label{Lem-dth-coeff}
Let $\epsilon>0$ and $d\in\Z_{\geq0}$. Then there exists a linear polynomial $L$ such that
$$[u^d]H_0(u;B) = L(d)q^d+ O_{\epsilon}\big(q^{(\frac{1}{2}+\epsilon)d}\big).$$
Furthermore, there exists a constant $C$ such that
$$[u^d]H_1(u;B) = Cq^d+O_{\epsilon}\big(q^{(\frac{1}{2}+\epsilon)d}\big).$$
\end{lem}

\begin{proof}
We have
\begin{align*}
    H_j(u;B) & = \prod_{Q\nmid B} \left(1+(\xi_3^ju)^{\deg(Q)}+(\xi_3^{2j}u)^{\deg(Q)}\right) \\
    & = \prod_{Q}\left(1-(\xi_3^ju)^{\deg(Q)}\right)^{-1}\prod_{Q}\left(1-(\xi_3^{2j}u)^{\deg(Q)}\right)^{-1} \widetilde{H}_j(u;B)\\
    & = \frac{1}{\big(1-\xi_3^jqu\big)\big(1-\xi_3^{2j}qu\big)} \widetilde{H}_j(u;B),
\end{align*}
where 
$$\widetilde{H}_j(u;B) = \prod_{Q\mid B}\left(1+O\big(u^{\deg(Q)}\big)\right)\prod_{Q\nmid B}\left(1+O\big(u^{2\deg(Q)}\big)\right)$$ 
is absolutely convergent in the region $|u|<q^{-1/2}$. Let $\epsilon>0$ and  $\Gamma= \{u : |u| = q^{-1/2-\epsilon}\}$. Then $\frac{H_j(u;B)}{u^{d+1}}$ is meromorphic in the region bounded by $\Gamma$ with poles at $u=0$ and $u=q^{-1}$ if $j=0$, and with poles at $u=0$, $u=\xi_3q^{-1}$ and $u=\xi_3^2q^{-1}$ if $j=1,2$. Hence,
$$\frac1{2\pi i}\oint_{\Gamma} \frac{H_j(u;B)}{u^{d+1}}\,du \ll \max_{u\in \Gamma} \left(\frac{|H_j(u;B)|}{|u|^{d}}\right) \ll_{\epsilon}  q^{(\frac{1}{2}+\epsilon)d}.$$

On the other hand, if $j=0$, then
$$\frac1{2\pi i}\oint_{\Gamma} \frac{H_0(u;B)}{u^{d+1}}\,du = \Res_{u=0}\left(\frac{H_0(u;B)}{u^{d+1}}\right) + \Res_{u=q^{-1}}\left(\frac{H_0(u;B)}{u^{d+1}}\right). $$
Moreover, we note that
$$\Res_{u=0}\left(\frac{H_0(u;B)}{u^{d+1}}\right) = [u^d]H_0(u;B)$$
and
\begin{align*}
    \Res_{u=q^{-1}}\left(\frac{H_0(u;B)}{u^{d+1}}\right) & = \lim_{u\to q^{-1}} \frac{d}{du} \left( \frac{(u-q^{-1})^2}{(1-qu)^2} \frac{\widetilde{H}_0(u;B)}{u^{d+1}}  \right) \\
    & = -\big((d+1)\widetilde{H}_0(q^{-1};B) - q^{-1}\widetilde{H}'_0(q^{-1};B)\big)q^d,
\end{align*}
which proves the result for $j=0$. Finally, if $j=1$, then we have 
\begin{align*}\frac1{2\pi i}\oint_{\Gamma} \frac{H_1(u;B)}{u^{d+1}}\,du & = \Res_{u=0}\left(\frac{H_1(u;B)}{u^{d+1}}\right) \\ 
& + \Res_{u=\xi_3q^{-1}}\left(\frac{H_1(u;B)}{u^{d+1}}\right)+ \Res_{u=\xi^2_3q^{-1}}\left(\frac{H_1(u;B)}{u^{d+1}}\right),
\end{align*}
where
$$\Res_{u=0}\left(\frac{H_1(u;B)}{u^{d+1}}\right) = [u^d]H_1(u;B)$$
and
\begin{align*}
    \Res_{u=\xi_3q^{-1}}\left(\frac{H_1(u;B)}{u^{d+1}}\right) & = \lim_{u\to \xi_3q^{-1}}  \frac{(u-\xi_3q^{-1})\widetilde{H}_1(u;B)}{(1-\xi_3qu)(1-\xi^2_3qu)u^{d+1}}\\
    & = -\bigg(\frac{\widetilde{H}_1(\xi_3q^{-1};B)}{(1-\xi_3^2)\xi_3^d}\bigg)q^d = -C_1q^d.
\end{align*}
By a similar calculation, we find that the residue at $\xi^2_3q^{-1}$ equals $-C_2q^d$ for a suitable constant $C_2$. Setting $C=C_1+C_2$ completes the proof.
\end{proof}

Lemma \ref{Lem-dth-coeff} has the following immediate consequence.

\begin{cor}\label{FNBsize}
Let $L$ and $C$ be as in Lemma \ref{Lem-dth-coeff}. Then, we have 
$$|\F_N(B)| = \frac{1}{3}\big(L(N)+2C\big)q^N + O_{\epsilon}\big(q^{(\frac{1}{2}+\epsilon)N}\big).$$
\end{cor}

\begin{proof}
Using Lemma \ref{Lem-dth-coeff}, we obtain
\begin{align*}
    |\F_N(B)| & = [u^N]\mathcal{G}(u;B)  = \frac{1}{3} [u^N]\big(H_0(u;B)+2H_1(u;B)\big)\\
    & = \frac{1}{3}\big(L(N)+2C\big)q^N + O_{\epsilon}\big(q^{(\frac{1}{2}+\epsilon)N}\big),
\end{align*}
as desired.
\end{proof}

In order to prove Proposition \ref{CubicCoprimeProp}, we need to know how to pass from $[u^d]\mathcal{G}(u;B)$ to $[u^d]\mathcal{G}(u;PB)$. Fortunately, using the fact that $H_j(u;PB)$ has an Euler product, it is straightforward to pass from $[u^d]H_j(u;B)$ to $[u^d]H_j(u;PB)$.

\begin{lem}\label{Lem-H-dth-coeff}
Let $d\in\Z_{\geq0}$ and let $P$ be a prime of degree $m$ such that $P\nmid B$. Then
$$[u^d]H_j(u;PB) = \sum_{a=0}^{\lfloor\frac{d}{m}\rfloor}(-1)^a \left(\zeta_3^{jm}+\zeta_3^{2jm}\right)^a [u^{d-am}]H_j(u;B).$$
\end{lem}

\begin{proof}
From the definition of $H_j(u;B)$, we see that
$$H_j(u;B) = \left(1+(\xi_3^ju)^m+(\xi_3^{2j}u)^m\right)H_j(u;PB)$$
and it follows that 
$$[u^d]H_j(u;B) = [u^d]H_j(u;PB) + \left(\zeta_3^{jm}+\zeta_3^{2jm}\right)[u^{d-m}]H_j(u;PB).$$
Rearranging, we then get
\begin{multline*}
    [u^d]H_j(u;PB) = [u^d]H_j(u;B) -\left(\zeta_3^{jm}+\zeta_3^{2jm}\right)[u^{d-m}]H_j(u;PB)\\
    = [u^d]H_j(u;B) -\left(\zeta_3^{jm}+\zeta_3^{2jm}\right)[u^{d-m}]H_j(u;B) 
    + \left(\xi_3^{jm}+\xi_3^{2jm}\right)^2 [u^{d-2m}]H_j(u;PB),
\end{multline*}
and iterating this procedure, we obtain 
\begin{align*}
    [u^d]H_j(u;PB) = \sum_{a=0}^{\lfloor\frac{d}{m}\rfloor} (-1)^a\left(\zeta_3^{jm}+\zeta_3^{2jm}\right)^a [u^{d-am}]H_j(u;B).
\end{align*}
This concludes the proof of the lemma.
\end{proof}

Using Lemmas \ref{Lem-dth-coeff} and \ref{Lem-H-dth-coeff}, we now complete the proof of Proposition \ref{CubicCoprimeProp}. 

\begin{proof}[Proof of Proposition \ref{CubicCoprimeProp}]
It follows from Lemmas \ref{Lem-dth-coeff} and \ref{Lem-H-dth-coeff} that if $P\nmid B$, then
\begin{align*}
    |\F_N(PB)| & = [u^N]\mathcal{G}(u;PB) = \frac{1}{3}[u^N]\big(H_0(u;PB) + 2H_1(u;PB)\big) \\ 
    & \hspace{-30pt} = \frac{1}{3}\sum_{a=0}^{\lfloor\frac{N}{m}\rfloor}(-1)^a [u^{N-am}]\left( 2^a H_0(u;B) + 2\left(\zeta_3^m+\zeta_3^{2m}\right)^a H_1(u;B) \right)\\
    & \hspace{-30pt} = |\F_N(B)| + \frac{1}{3} \sum_{a=1}^{\lfloor\frac{N}{m}\rfloor} (-1)^a \left(2^a L(N-am) + 2C\left(\zeta_3^m+\zeta_3^{2m}\right)^a \right)q^{N-am} 
    + O_{\epsilon}\big(q^{(\frac{1}{2}+\epsilon)(N-m)}\big).
\end{align*}
Hence, by Corollary \ref{FNBsize}, we find that $\frac{|\F_N(PB)|}{|\F_N(B)|}$ equals
\begin{align*}
    1 & + \frac{\frac{1}{3}\sum_{a=1}^{\lfloor\frac{N}{m}\rfloor} (-1)^a \left(2^a L(N-am)+2C\left(\zeta_3^m+\zeta_3^{2m}\right)^a \right)q^{N-am} +O_{\epsilon}\big(q^{(\frac{1}{2}+\epsilon)(N-m)}\big)}{\frac{1}{3}\big(L(N)+2C\big)q^N + O_{\epsilon}\big(q^{(\frac{1}{2}+\epsilon)N}\big)}\\
    & = 1 + \sum_{a=1}^{\lfloor\frac{N}{m}\rfloor} \left(\frac{-2}{q^m}\right)^a\left(1-\frac{am}{N}\right) + O\left(\frac{1}{Nq^m}\right),
\end{align*}
which is the desired result. 
\end{proof}

We conclude this subsection with an observation that is useful in the proof of Lemma \ref{BadPrimeBound}.

\begin{rem}\label{Sec2Rem}
We can use the same ideas as above to prove that for any prime $P$, we have
\begin{align*}
    & |\{ D\in \F_N(B) : P|D \}| \\ 
    & \hspace{20pt} \leq 2\left|\big\{ D\in \Ff_q[T]: D \mbox{ monic, cube-free and } \deg(\rad(D))=N-\deg(P) \big\}\right|\\
    & \hspace{20pt} \ll (N-\deg(P)) q^{N-\deg(P)}   \ll \frac{|\F_N(B)|}{q^{\deg(P)}}.
\end{align*} 
Indeed, the first inequality is obvious as we have removed the conditions of being coprime to $B$ and $\deg(D)\equiv 0\bmod{3}$, and the factor $2$ accounts for the two cases $P\|D$ and $P^2\|D$. The second bound follows by the same proof as that of Corollary \ref{FNBsize} with the minor change that here we only need to consider $H_0(u;1)$ as we have removed the conditions $(D,B)=1$ and $\deg(D)\equiv 0\bmod{3}$. 
\end{rem}

\subsection{Proof of Theorem \ref{CubicLowTerm}}

We now have everything we need to complete the proof of Theorem \ref{CubicLowTerm}.

\begin{proof}[Proof of Theorem \ref{CubicLowTerm}]
Recall from \eqref{CubicTrForm2} that
\begin{align*}
    \big\langle\Tr(\Theta_{\widetilde{E}_D}^n)\big\rangle_{\F_N(B)} = M(n,N) + S(n,N) + E(n,N) + O\left(\frac{1}{q^{3n/8}} + \frac{n(\deg(\Delta)+\tau(n))}{q^{n/2}}\right).
\end{align*}
Using Proposition \ref{EPprop} to bound $E(n,N)$, we get
\begin{multline*}
    E(n,N) = \frac{n}{q^{n/2}} \sum_{\deg(P)=n} \big(\lambda_P\overline{E}_P+\overline{\lambda}_PE_P\big)  - \frac{n/2}{q^{n/2}} \sum_{\deg(P)=\frac{n}{2}} \big(\lambda_P^2E_P+\overline{\lambda}^2_P\overline{E}_P\big)  \\
    + \frac{n/3}{q^{n/2}} \sum_{\deg(P)=\frac{n}{3}} \ 3(|\lambda_P|^2-1)\big(\lambda_P\overline{E}_P+\overline{\lambda}_PE_P\big) \ll_{\epsilon} \frac{q^{n/2}e^{2n}}{Nq^{(\frac{1}{2}-\epsilon)N}}.
\end{multline*}
Finally, we use Proposition \ref{CubicCoprimeProp} to estimate $M(n,N)$ and $S(n,N)$. We obtain
\begin{align*}
    M(n,N) & = -\frac{n/2}{q^{n/2}}  \sum_{\deg(P)=\frac{n}{2} } 2(|\lambda_P|^2-1)  \frac{|\F_N(PB)|}{|\F_N(B)|} \\
    & =  -\frac{n/2}{q^{n/2}}  \sum_{\deg(P)=\frac{n}{2}} 2(|\lambda_P|^2-1) + \frac{\widetilde{\mathcal D}_1(n)}{q^{n/2}} + O\left(\frac{1}{Nq^{n/2}}\right),
\end{align*}
where
\begin{align}\label{CubicDForm1}
\widetilde{\mathcal D}_1(n) :=  -\frac{n}{2} \sum_{\substack{\deg(P)=\frac{n}{2} \\ P\nmid B}} 2(|\lambda_P|^2-1) \sum_{a=1}^{\lfloor\frac{2N}{n}\rfloor} \left(\frac{-2}{q^{n/2}}\right)^a\left(1-\frac{an}{2N}\right) \ll 1,
\end{align}
and
\begin{align*}
    S(n,N)  & = \frac{n/3}{q^{n/2}}  \sum_{ \deg(P)=\frac{n}{3}} \big(\lambda^3_P + \overline{\lambda}^3_P\big) \frac{|\F_N(PB)|}{|\F_N(B)|} \\
    & = \frac{n/3}{q^{n/2}}  \sum_{ \deg(P)=\frac{n}{3}} \big(\lambda^3_P + \overline{\lambda}^3_P\big) + \frac{\widetilde{\mathcal{D}}_2(n)}{q^{n/2}} + O\left(\frac{1}{Nq^{n/2}}\right), 
\end{align*}
where 
\begin{align}\label{CubicDform2}
    \widetilde{\mathcal{D}}_2(n):= \frac{n}{3}  \sum_{\substack{ \deg(P)=\frac{n}{3} \\ P\nmid B}} \big(\lambda^3_P + \overline{\lambda}^3_P\big) \sum_{a=1}^{\lfloor\frac{3N}{n}\rfloor} \left(\frac{-2}{q^{n/3}}\right)^a\left(1-\frac{an}{3N}\right)\ll 1.
\end{align}
This completes the proof.
\end{proof}

\section{Heuristics and conjectures}\label{Section:heuristics}

In this section, we return to the heuristic arguments in Section \ref{Intro-Heur-Conj}. We first prove Corollary \ref{maincor}. 

\begin{proof}[Proof of Corollary \ref{maincor}]
First, we note that 
$$\frac{n/3}{q^{n/2}}\sum_{\deg(P)=\frac{n}{3}} \big(\lambda_P^3+\overline{\lambda}_P^3\big) \ll \frac{1}{q^{n/6}}.$$
Hence, setting the right-hand sides of Theorems \ref{CubicOLD} and \ref{CubicLowTerm} equal to each other for $n=2m$, we find that
\begin{align}\label{Lambda2-first} 
-\frac{m}{q^m} \sum_{\deg(P)=m}2(|\lambda_P|^2-1) & = 1 - \frac{2m/3}{q^m} \sum_{\deg(P)=\frac{2m}{3}} \big(\lambda_P^3 + \overline{\lambda}^3_P\big) + O\left(\frac{m}{q^{m/2}} \right) \nonumber\\
& = 1 + O\left(\frac{1}{q^{m/3}}\right).
\end{align}
Furthermore, by the Prime Polynomial Theorem, we get
\begin{align}\label{Lambda2-second}
-\frac{m}{q^m} \sum_{\deg(P)=m}2(|\lambda_P|^2-1) = -\frac{2m}{q^m} \sum_{\deg(P)=m} |\lambda_P|^2 +2 + O\left(\frac{m}{q^{m/2}}\right).
\end{align}
Finally, equating the right-hand sides of \eqref{Lambda2-first} and \eqref{Lambda2-second}, we find that
$$\frac{m}{q^m} \sum_{\deg(P)=m} |\lambda_P|^2 = \frac{1}{2} + O\left(\frac{1}{q^{m/3}}\right),$$
as desired.
\end{proof}

Next, we turn our attention to Conjecture \ref{CubicConj}. For evidence of the conjecture, we first consider the term
$$\frac{n/3}{q^{n/2}}\sum_{\deg(P)=\frac{n}{3}} \big(\lambda_P^3+\overline{\lambda}_P^3\big)$$
from the right-hand side of Theorem \ref{CubicLowTerm}. For primes of good reduction, we use \eqref{amPkform} and Lemma \ref{atolambda} to write
\begin{align*}
    a^*_{3,P} & = \alpha_P^3+\beta_P^3 + \alpha_P+\beta_P \\
    & = a^*_{1,P^3} + a^*_{1,P} \\
    & = -\big(\lambda_P^3+\overline{\lambda}_P^3\big) - 3(|\lambda_P|^2-1)(\lambda_P + \overline{\lambda}_P) + a^*_{1,P} \\
    & = -\big(\lambda_P^3+\overline{\lambda}_P^3\big) + (3|\lambda_P|^2-2)a^*_{1,P}.
\end{align*}
Hence, if we replace $|\lambda_P|^2$ with its average value when averaging over $\deg(P)=\frac{n}{3}$ and assume that the contribution of the primes of bad reduction will be of lower order, then we get
\begin{align*}
    \frac{n/3}{q^{n/2}}\sum_{\deg(P)=\frac{n}{3}} \big(\lambda_P^3+\overline{\lambda}_P^3\big) & \approx -\frac{n/3}{q^{n/2}}\sum_{\deg(P)=\frac{n}{3}} a^*_{3,P} +  \frac{n/3}{q^{n/2}}\sum_{\deg(P)=\frac{n}{3}} (3|\lambda_P|^2-2)a^*_{1,P}\\
    & \approx  -\frac{n/3}{q^{n/2}}\sum_{\deg(P)=\frac{n}{3}} a^*_{3,P} - \frac{1}{2} \frac{n/3}{q^{n/2}}\sum_{\deg(P)=\frac{n}{3}} a^*_{1,P} \\
    & \sim \frac{\eta_3(n)}{q^{n/3}} \left( \Tr\big(\Theta^{n/3}_{sym^3\widetilde{E}}\big) + \frac{1}{2}\Tr\big(\Theta^{n/3}_{\widetilde{E}}\big) \right).
\end{align*}
In the final step above, we applied \eqref{SymmEllTraceForm} together with the assumption that the sums in the second line constitute the dominant contribution to the respective trace.

Finally, by a similar argument we find that 
\begin{align*}
    -\frac{n/2}{q^{n/2}} \sum_{\deg(P)=\frac{n}{2}} 2(|\lambda_P|^2-1) 
    \approx \frac{n/2}{q^{n/2}} \sum_{\deg(P)=\frac{n}{2}}1 \sim \eta_2(n).
\end{align*}
Assuming that the tertiary main term is handled in the way described in Remark \ref{Remark:tertiarymainterm} and that all error terms are sufficiently small, we arrive at Conjecture \ref{CubicConj}.

\appendix
\section{Comments on families of twists}\label{Section:comments}

One criterion for a family of function field $L$-functions to be considered ``nice" and thus have interesting symmetries is that all members of the family have the same degree of the conductor, or equivalently, that all $L$-functions in the family are polynomials of the same degree. Ideally, when considering a certain class of $L$-functions, one would like the family to contain all the $L$-functions in that class of a given degree.

For the cubic twists, we consider only the family $\F_N(B)$ and all $L$-functions of twists of $\widetilde{E}$ by a polynomial $D\in\F_N(B)$ have degree $\mathfrak{n}+2N$. However, since the condition $\deg(D)\equiv 0 \bmod{3}$ is included in the definition of $\F_N(B)$, we clearly see that there are cubic twists of $\widetilde{E}$ that do not come from $\F_N(B)$ and, indeed, some of them will have $L$-functions of degree $\mathfrak{n}+2N$. 

It is easy to check that every \textit{finite} prime that divides $D$ will have additive reduction on $\widetilde{E}_D$ (or $E_D$). Each of these primes will then contribute $P^2$ to the conductor of $\widetilde{E}_D$ and so the degree of $L(u,\widetilde{E}_D)$ will always be roughly $\mathfrak{n}+2\deg(\rad(D))$. It remains to check what happens to the prime at infinity.

\subsection{Reduction at the prime at infinity}

Let $E:y^2=x^3+Ax+B$ with $A,B\in\Ff_q[T]$. Define $a:=\deg(A)$ and $b:=\deg(B)$. To analyze what happens at the point at infinity, we let $S:=1/T$ and write
$$A(T) = S^{-a}A^*(S), \quad \quad \quad B(T) := S^{-b}B^*(S)$$
and analyze what happens at $S=0$ for the curve given by the equation
$$y^2 = x^3 +S^{-a}A^*x + S^{-b}B^*. $$
Setting 
$$\ell = \max\left\{\lceil a/2\rceil, \lceil b/3\rceil \right\}$$
and
$$Y = S^{3\ell/2}y, \quad \quad X = S^{\ell}x,$$
we rewrite the equation as 
\begin{align}\label{SmoothInfModel}
    Y^2 = X^3 + S^{2\ell-a}A^*X +S^{3\ell-b}B^*.
\end{align}
We now note that 
$$A^*(0),B^*(0)\not=0$$ 
unless $A=0$ or $B=0$. Moreover, we have
$$2\ell-a,3\ell-b\geq 0$$ 
and either 
$$2\ell-a\leq 1 \qquad \mbox{ or }\qquad 3\ell-b\leq 2.$$
Hence, \eqref{SmoothInfModel} is a minimal Weierstrass equation for $S=0$. Therefore, if $P_{\infty}$ is the prime at infinity, then we define
\begin{align}\label{EPinfinity}
E(P_{\infty}) : Y^2 = X^3 + S^{2\ell-a}A^*X +S^{3\ell-b}B^* \mod{S}.    
\end{align}

We are now ready to discuss the reduction type for the prime at infinity. We get several cases:

\begin{enumerate}
\item If $a/2<b/3$, then $2\ell-a>0$. Hence, reducing mod $S$ in \eqref{EPinfinity}, we get:
\begin{enumerate}
    \item $3|b$ : $Y^2=X^3+B^*(0)$, so $P_{\infty}$ is a prime of \textbf{good reduction};
    \item $3\nmid b$: $Y^2=X^3$, so $P_{\infty}$ is a prime of \textbf{additive reduction}.
\end{enumerate}
\vspace{3pt}
\item If $a/2>b/3$, then $3\ell-b>0$. Hence, reducing mod $S$ in \eqref{EPinfinity}, we get:
\begin{enumerate}
    \item $2|a$ : $Y^2=X^3+A^*(0)X$, so $P_{\infty}$ is a prime of \textbf{good reduction};
    \item $2\nmid a$: $Y^2=X^3$, so $P_{\infty}$ is a prime of \textbf{additive reduction}.
\end{enumerate}
\vspace{3pt}
\item If $a/2=b/3$, then it must be that $2|a$ and $3|b$. Hence, reducing mod $S$ in \eqref{EPinfinity}, we get $Y^2=X^3+A^*(0)X+B^*(0)$. 
\begin{enumerate}
    \item If $4A^*(0)^3+27B^*(0)^2\not=0$, then $P_{\infty}$ is a prime of \textbf{good reduction}.
    \item If $4A^*(0)^3+27B^*(0)^2=0$, then $P_{\infty}$ is a prime of \textbf{multiplicative reduction}.
\end{enumerate}
\end{enumerate}

\subsection{The quadratic twist family}

In the case of quadratic twists, we have
$$E : y^2=x^3+Ax+B$$
and
$$E_D : y^2 = x^3+AD^2x+BD^3.$$
We observe that the conditions on the degrees (from the cases at the end of the previous subsection) aren't changed by twisting by $D\in \Hh^{\pm}_N(\Delta)$. Moreover, in the case that $a/2=b/3$, we get
$$4\big((AD^2)^*(0)\big)^3 + 27\big((BD^3)^*(0)\big)^2 = 4A^*(0)^3+27B^*(0)^2$$
since $D^*(0)=1$. Therefore, the reduction at the prime at infinity is not changed by quadratic twists by $D\in \Hh^{\pm}_N(\Delta)$ and thus $\Hh^{\pm}_N(\Delta)$ is a ``nice" and ``full" family in the above sense.

\subsection{The cubic twist family}

In the case of cubic twists, we have
$$\widetilde{E}: y^2 = x^3 + B$$
and
$$\widetilde{E}_D : y^2 = x^3+BD^2.$$
Therefore, we see that since $A=0$ we always have $a/2<b/3$ and hence are only ever in case $(1)$ above. The following chart captures, for $D$ cube-free and coprime to $B$, what happens to the prime at infinity and how this affects the degree of the conductor:

\vspace{3pt}
\begin{center}
\begin{tabular}{c|c|c|c|c}
    $b\bmod{3}$ & $\deg(D) \bmod{3}$ & Ram. for $\widetilde{E}$ &  Ram. for $\widetilde{E}_D$ & $\mathfrak{n}_{\widetilde{E}_D}-\mathfrak{n}_{\widetilde{E}}$\\
    \hline
     $0$ & $0$ & good & good &  $2\deg(\rad(D))$\\
     $0$ & $\not\equiv 0$ & good & additive & $2(\deg(\rad(D))+1)$\\
     $\not\equiv 0$ & $b$  & additive & good & $2(\deg(\rad(D))-1)$\\
     $\not\equiv{0}$ & $\not\equiv b$ & additive & additive & $2\deg(\rad(D))$
\end{tabular}
\end{center}
\vspace{3pt}

As a consequence, to create a family of cubic twists in which all curves have the same degree of the conductor, we define the set
$$\widehat{\F}_N(B) := \big\{D\in\Ff_q[T] : D \mbox{ monic, cube-free}, (D,B)=1, \deg(\rad(D))=N \big\}.$$
As we see above, the ramification at infinity depends only on the congruence of $\deg(D)\mod{3}$. Hence, we define
$$\widehat{\F}_{N;k}(B) := \big\{D\in\widehat{\F}_N(B) : \deg(D)\equiv k \bmod{3} \big\},$$
$$\mathcal{J}_N(B) := \widehat{\F}_{N;b}(B) \cup \widehat{\F}_{N-1,b+1}(B) \cup \widehat{\F}_{N-1,b+2}(B),$$
and
$$\mathcal{K}_N(B):=\begin{cases} \mathcal{J}_N(B) & 3|b, \\ \mathcal{J}_{N+1}(B) & 3\nmid b. \end{cases}$$
From all this, we may now conclude that the family
$$\big\{\widetilde{E}_D : D\in \mathcal{K}_N(B)\big\}$$
consists of all cubic twists of $\widetilde{E}$ with $L$-functions of degree $\mathfrak{n}+2N$. 

Finally, we note that $\widehat{\F}_{N;0}(B)=\F_{N}(B)$. It is relatively easy to see how one would adapt our methods in Section \ref{Section:cubic} to deal with the ``full" family $\mathcal{K}_N(B)$. One would just have to split everything into the relevant cases and the same proofs would work. Hence, we leave this as a comment and end the discussion here.

\end{document}